\documentclass[11pt,notitlepage,a4paper,reqno,ps]{amsart}
\usepackage{eurosym}
\usepackage{amssymb}
\usepackage{amstext}
\usepackage{color}
\usepackage{xcolor}
\usepackage{amsmath,amsthm,amssymb,amscd,epsfig,esint, mathrsfs,graphics,color}
 \usepackage{wrapfig}
\usepackage{verbatim}
\usepackage[latin1]{inputenc}
\xdefinecolor{azzurro}{rgb}{0.2,0.9,0.9}%azzurro
\xdefinecolor{violetto}{rgb}{0.5,0.5,0.9}%violetto
\xdefinecolor{marrone}{rgb}{0.8,0.4,0.2}%marroncino

\newcommand{\medint}{-\kern  -,375cm\int}

\definecolor{ora}{rgb}{0.8,0.2,0.1}
\definecolor{vio}{rgb}{0.5,0,0.5}
\definecolor{gre}{rgb}{0.1,0.6,0}
\definecolor{verde}{rgb}{0,0.7,0.4}
\newenvironment{michelarev}{\color{azzurro}}{\color{black}}
\newcommand{\bmicr}{\begin{michelarev}}
\newcommand{\emicr}{\end{michelarev}}
\theoremstyle{plain}
\newtheorem{theorem}{Theorem}[section]
\newtheorem{corollary}[theorem]{Corollary}
\newtheorem{lemma}[theorem]{Lemma}

\theoremstyle{definition}

\theoremstyle{remark}
\newtheorem{remark}[theorem]{Remark}
\theoremstyle{plain}

\def\tcr{\textcolor{red}} %correzioni
\def\cA{\mathcal{A}}
\def\R{\mathbb{R}}
\numberwithin{equation}{section} \makeatletter

\renewcommand{\p@enumi}{\thesection.}
\makeatother  \allowdisplaybreaks
\email{mcaselli@student.ethz.ch}
\email{michela.eleuteri@unimore.it}
  \email{antonia.passarellidinapoli@unina.it}

\begin{document}
\title[Regularity results for a class of obstacle problems with $p,q-$growth conditions]{Regularity results for a class of obstacle problems with $p,q-$growth conditions}
\author[M. Caselli -- M. Eleuteri -- A. Passarelli di Napoli]{Michele Caselli -- Michela Eleuteri -- Antonia Passarelli di Napoli}
\address{ETH Z\"urich, Department of Mathematics, R\"amistrasse 101, 8092 Z\"urich, Switzerland}
\address{Dipartimento di Scienze Fisiche, Informatiche e Matematiche, Universit\`a degli Studi di Modena E Reggio Emilia,
via Campi 213/b, 41125 Modena, Italy}
\address{Dipartimento di Matematica e Applicazioni ``R. Caccioppoli''
\\
Universit\`a degli Studi di Napoli
\\
``Fede-rico II''
Via Cintia, 80126, Napoli, Italy}
\thanks{\textit{Acknowledgements.}
The work of the authors is supported by GNAMPA (Gruppo Nazionale per l'Analisi Matematica, la Probabilit\`a e le loro Applicazioni) of INdAM (Istituto Nazionale di Alta Matematica), by MIUR through the project FFABR and
%The work of is supported by GNAMPA (Gruppo Nazionale per l'Analisi Matematica, la Probabilit\`a e le loro Applicazioni) of INdAM (Istituto Nazionale di Alta Matematica), through the projects GNAMPA 2016 ``Problemi di
Regolarit\`a nel Calcolo delle Variazioni e di
%Approssimazione'' (coord. Prof. M. Carozza) and GNAMPA 2017 ``Approssimazione con operatori discreti e problemi di minimo per funzionali del calcolo delle variazioni con applicazioni all'imaging'' (coord. Dott. D.
Costarelli).  The work of the authors is also supported
by the University of Modena and Reggio Emilia through the project FAR2017 ``Equazioni differenziali: problemi evolutivi, variazionali ed applicazioni'' (coord. Prof. S. Gatti). This research was performed while A. Passarelli
di Napoli was visiting the University of Modena and Reggio Emilia and M. Caselli and M. Eleuteri were visiting the University of Naples ``Federico II". The hospitality of both Institutions is warmly adknowledged}.

\begin{abstract}
In this paper we prove the  the local Lipschitz continuity for solutions  to a class of obstacle problems of the type
$$\min\left\{\int_\Omega {F(x, Dz)}: z\in \mathcal{K}_{\psi}(\Omega)\right\}.$$ Here $\mathcal{K}_{\psi}(\Omega)$ is the set of admissible functions $z \in {u_0 + W^{1,p}(\Omega)}$ {for a given $u_0 \in W^{1,p}(\Omega)$}
such that $z \ge \psi$ a.e. in $\Omega$, $\psi$ being the obstacle and $\Omega$ being an open bounded set of $\mathbb{R}^n$, $n \ge 2$.
The main novelty here is that we are  assuming that the integrand $ F(x, Dz)$ satisfies $(p,q)$-growth conditions and as a function of the $x$-variable belongs to a suitable Sobolev class. Moreover, we impose less
restrictive assumptions on the obstacle with respect to the existing regularity results.
 Furthermore, assuming the obstacle $\psi$ is locally bounded, we prove the local boundedness of the solutions to a quite large class of variational inequalities whose principal part satisfies non standard growth
 conditions.
\end{abstract}

\maketitle

\begin{center}
\fbox{\today}
\end{center}

\section{Introduction}

The aim of this paper is the study of the local Lipschitz continuity of the solutions to a class of variational obstacle problems of the form
 \begin{equation}
\label{obst-def0}
\min\left\{\int_\Omega F(x, Dz): z\in \mathcal{K}_{\psi}(\Omega)\right\},
\end{equation}
where $\Omega$ is a bounded open set of $\mathbb{R}^n$, $n \ge 2$.
The function $\psi:\,\Omega \rightarrow [- \infty, + \infty)$, called \textit{obstacle}, belongs to the Sobolev class $W^{1,p}(\Omega)$ and the class $\mathcal{K}_{\psi}(\Omega)$ is defined as follows
\begin{equation}
\label{classeA}
\mathcal{K}_{\psi}(\Omega) := \left \{z \in u_0+{W^{1,p}_0(\Omega)}: z \ge \psi \,\, \textnormal{a.e. in $\Omega$} \right\},
\end{equation}
where $u_0\in W^{1,p}(\Omega)$ is a fixed boundary value.
To avoid trivialities, in what follows we shall assume that $\mathcal{K}_{\psi}$ is not empty and that  a solution $u$ to \eqref{obst-def0} is such that $F(x,Du)\in L^1_{\mathrm{loc}}(\Omega)$.

We shall consider integrands $F$ such that $\xi \mapsto F(x, \xi)$ is $\mathcal{C}^2$ and   there exists $f:\Omega\times [0,+\infty)\mapsto [0,+\infty)$ such that $F(x, \xi)=f(x,|\xi|)$.  We shall assume the following set of
conditions:

$$  \tilde{\nu} {(1 + |\xi|^2)^{\frac{p}{2}}} \le \, F(x,\xi) \le \, \tilde{L} {(1 + |\xi|^2)^{\frac{q}{2}}} \eqno{\mathrm{(F1)}}$$
$$ \tilde{\nu}_1 (1 + |\xi|^2)^{\frac{p-2}{2}} |\lambda|^2 \le \, \langle F_{\xi\xi}(x,\xi) \lambda, \lambda \rangle
\le \, \tilde{L}_1 (1 + |\xi|^2)^{\frac{q-2}{2}} |\lambda|^2 \eqno{\mathrm{(F2)}}$$
$$ |{F_{x \xi}}(x, \xi) | \le \, h(x) {(1 + |\xi|^2)}^{\frac{q-1}{2}} \eqno{\mathrm{(F3)}}$$
for almost all $x \in \Omega$,  and all $\xi, \lambda \in \mathbb{R}^n$, where $ 2\le p\le q$ and $0 \le \tilde{\nu} \le \tilde{L}$, $0 \le \tilde{\nu}_1 \le \tilde{L}_1$  are fixed constants,  and  $h(x)$ is a fixed non
negative function.
%such that $h(x)\in W^{1,r}$, some $r>n$.
\\ In case of standard growth conditions, i.e. when $p=q$ in (F1)--(F3), it is well known that    $u \in W^{1,p}(\Omega)$ is a {\it solution to the obstacle problem in $\mathcal{K}_{\psi}(\Omega)$}  if and only if $u\in
\mathcal{K}_\psi$ solves the following variational inequality
\begin{equation}
\label{obst-def}
\int_{\Omega} \langle \mathcal{A}(x, Du), D(\varphi - u) \rangle \, dx \ge 0,
\end{equation}
 for all $\varphi \in \mathcal{K}_{\psi}(\Omega)$,  where we set $$\mathcal{A}(x,\xi)= {F_{\xi}}(x,\xi).$$
 Here, dealing with non standard growth, it is worth observing that \eqref{obst-def} holds for solutions to \eqref{obst-def0} assuming a suitable closeness condition between the growth and the ellipticity exponents.
 \\
 Actually, already for non constrained problems with non standard growth conditions, the relation  between minima and extremals, i.e. solutions of the corresponding Euler Lagrange system, is an issue that requires a careful
 investigation (see for example \cite{CKPPisa, CKP Comm}).

 Due to our assumptions on the gap, see \eqref{gap} below,  and by virtue of the regularity result proven in \cite{G19}, the validity of \eqref{obst-def} can be easily checked {(see Subsection \ref{EL}).}
 \\
Note that assumptions (F1)--(F3) imply that
%$$ \langle \mathcal{A}(x, \xi) - \mathcal{A}(x, \eta), \xi - \eta \rangle  \ge \, \nu |\xi - \eta|^2 (\mu^2 + |\xi|^2 + |\eta|^2)^{\frac{p-2}{2}}\eqno{(\cA 1)}$$
%$$ |\mathcal{A}(x, \xi) - \mathcal{A}(x, \eta)| \le \, L \, |\xi - \eta| (\mu^2 + |\xi|^2 + |\eta|^2)^{\frac{q-2}{2}} \eqno{(\cA 2)}$$
$$ |\mathcal{A}(x, \xi)| \le \, \ell \, (1 + |\xi|^2)^{\frac{q-1}{2}} \eqno{(\cA 1)}$$
$$ \langle \mathcal{A}(x, \xi_1)- \mathcal{A}(x, \xi_2), \xi_1-\xi_2  \rangle  \ge \, \nu |\xi_1-\xi_2 |^2(1+|\xi_1|^2+|\xi_2|^2)^{\frac{p-2}{2}}\eqno{(\cA 2)}$$
$$ |{{\cA}_{x}}(x, \xi) | \le \, h(x) {(1 + |\xi|^2)}^{\frac{q-1}{2}} \eqno{\mathrm{(\cA3)}}$$
for a.e. $x\in \Omega$ and every $\xi,\xi_1,\xi_2\in \mathbb{R}^n$, where $0 \le \nu \le \ell$ are fixed constants.

%These assumptions are standard and express that the integrand $F$ is strictly convex with respect to the gradient variable $z$ and it is H\"older continuous with exponent $\alpha$ with respect to the variable $(x,y)$.  Note
that the parameter ${\tilde{\mu}}\in [0,1]$  allow us to consider  both the degenerate (${\tilde{\mu}}=0$) and the non-degenerate situation (${\tilde{\mu}}>0$).
%\\

{The study of the regularity theory for obstacle problems is a classical topic in Partial Differential Equations and Calculus of Variations and in the last years there has been an intense research activity for the regularity
of the obstacle problem in case of standard growth conditions: {in the recent paper \cite{BC19} a wide list of references has been provided
%among the others we quote \cite{BFM01}--\cite{BLS15}, \cite{C91}, \cite{CL91}, \cite{EH08}--\cite{EHL13},\cite{MZ86}--\cite{O17}} {
in the direction of H\"older continuity results in the spirit of the classical results} by De Giorgi, while we quote \cite{CF13}, \cite{FKO17}, \cite{PSU} inspired by the fundamental work by Caffarelli \cite{C76} and more
devoted to the analysis of the optimal regularity for the obstalcle problem together with the study of the regularity for the free boundary}.
\\
The general phenomenon observed in the above mentioned papers is that, in case of $p$-growth, the regularity of the solutions is influenced by the regularity of the obstacle. The first paper in which this phenomenon has been
observed to hold true also in case of non autonomous energy densities with $(p,q)$-growth is \cite{G19}, where an extra differentiability of $D\psi$ transfers into an extra differentiability of the gradient of the solutions.
\\
The aim of this paper is to proceed further in the investigation of the regularity properties of solutions to obstacle problems under $(p,q)$ growth condition.

When dealing with nonstandard growth, the first issue to deal with is the local boundedness of the solutions, since the known counterexamples to the regularity exhibit unbounded minimizers (see for example \cite{M87}) and
indeed it has been widely investigated  for solutions to PDE's and for minimizers of integral functionals (see for example {\cite{CMM1}--\cite{CMM3}} and the references therein).

Our first aim is to prove the local boundedness for solutions to a quite large class of obstacle problems with non standard growth conditions. {In \cite{CDF} a similar result ha been proved under more restrictive assumptions
on the obstacle and on the structure of the operator $\mathcal{A}$. Instead,} our first result shows that, under a suitable closeness condition on the exponents $p,q$, but without assuming any regularity on the partial map
$x\mapsto D_\xi F(x,\xi)$, solutions to \eqref{obst-def} are locally bounded provided the obstacle is bounded.
More precisely, we are going to prove the following
\begin{theorem}\label{locbound}
Let $u\in \mathcal{K}_\psi(\Omega)$ be a solution to \eqref{obst-def} under the assumptions $(\cA 1)$ and $(\cA 2)$ with $2\le p\le q$ such that
\begin{equation}\label{gap1}
	p\le q<p\frac{n-1}{n-p} \qquad\text{if}\qquad p<n.
\end{equation}
 If $\psi\in L^\infty_{\mathrm{loc}}(\Omega)$, then
 $u\in L^\infty_{\mathrm{loc}}(\Omega)$ and the following estimate
 \begin{eqnarray*}
 	\sup_{B_{R/2}}|u| \le c \left[{\sup_{B_R}} |\psi|+\left(\int_{B_R}|u|^{p^*}\,dx\right)\right]^\gamma
 \end{eqnarray*}
 holds for every ball $B_R\Subset\Omega$, for $\gamma(n,p,q)>0$ and with $c=c(\ell,\nu,p,q,n)$, {being $p^* = \frac{pn}{n-p}$.}
\end{theorem}
Notice that, as long as our regularity results have  local nature, we are not requiring further assumptions on the boundary datum $u_0$ in \eqref{classeA}.
\\
The proof of previous result is achieved arguing as in \cite{CMM1} i.e. using  the well known De Giorgi technique (see for example \cite[Chapter 7]{Giusti}) that consists in proving a suitable Caccioppoli inequality on the
superlevel sets of the solution $u$.
\\
In order to obtain such a Caccioppoli type inequality, one has to use test functions obtained truncating the solution. Here, we have to take into account that the test functions must belong to the admissible class
$\mathcal{K}_\psi$ and this is where the local boundedness of the obstacle $\psi$ comes into  play.
\\
As far as we know, the previous Theorem is new even in the case of standard growth, i.e. $p=q$.
\\
We'd like to mention that Theorem \ref{locbound} has been already employed in \cite{CGG} to establish an higher differentiability result under weaker assumptions on the integrand $F$ and on the obstacle $\psi$ with respect to
those in \cite{EPdN1}.
\\
Our second aim is to show that, under a suitable Sobolev regularity assumption on the gradient of the obstacle $\psi$ and on the partial  map $x\mapsto F_\xi(x,\xi)$, the solutions to \eqref{obst-def0} are  locally Lipschitz
continuous.
\\
In this regard, we have to mention that in \cite{FM00} such kind of regularity has been obtained  assuming nearly linear growth for autonomous Lagrangians $F\equiv F(\xi)$, through a linearization technique, that goes back to
\cite{F85}.
%, later refined in \cite{D87}, see also \cite{F90} -- \cite{FG}.}
The main idea of the linearization approach is to interpret the constrained minimizer, or more precisely the solution to \eqref{obst-def}, as a solution to an elliptic equation of the form
$$\mathrm{div}\mathcal{A}(x,Du)=g$$
with a right hand side,  obtained after the identification of a suitable Radon measure.
\\
%It is worth mentioning that the  higher differentiability properties for the solutions to previous equation have been widely investigated both in case of Sobolev and Besov coefficients (see for example \cite{BCGOP16},
\cite{CGP17}, \cite{G15}--\cite{GPdN}, \cite{PdN14-1}--\cite{PdN15}).

In the case of non autonomous functionals with standard growth conditions, the first result in this direction  is due to \cite{BC19}. In particular we remark that in \cite{FM00} a lot of effort has been employed to identify
the Radon measure and the authors explicitly say that this procedure could be significantly simplified if we would have a priori proved higher differentiability for local minimizers of the obstacle problem. But in a  recent
paper (\cite{EPdN1}), the authors were able  to establish the higher differentiability of integer and fractional order of the solutions to a class of obstacle problems (involving $p-$harmonic operators) assuming that the
gradient of the obstacle possesses an extra (integer or fractional) differentiability property.
The  use of this result  simplifies the procedure outlined in \cite{FM00} and has been employed in \cite{BC19} to obtain the local Lipschitz continuity of solutions to \eqref{obst-def} under standard growth conditions and
with Lipschitz continuous coefficients.

{In the case of non standard growth conditions the linearization argument has been performed  under a  $W^{2, \infty}$ assumption on the obstacle (\cite{DF19}). Actually, the higher differentiability of the solutions to
\eqref{obst-def0} has been obtained under weaker assumption   on the obstacle $\psi$ than $W^{2, \infty}$ (\cite{G19}) and this result allows us to argue as in \cite{BC19} to identify the right hand side of the elliptic
equation arising from the linearization technique. This is the crucial step for the local Lipschitz continuity of the solutions to \eqref{obst-def0}}.

{More precisely, our main result can be stated as follows:}

\begin{theorem}\label{loclip}
Let $u\in \mathcal{K}_\psi(\Omega)$ be a solution to {\eqref{obst-def0}} under the assumptions \textnormal{(F1)--(F3)}. Suppose that there exists $r>n$ such that $h\in L^r_{\mathrm{loc}}(\Omega)$, where $h(x)$ is the function
appearing in \textnormal{(F3)}. Assume moreover that  $2\le p\le q$ are such that
\begin{equation}\label{gap}
	p\le q<p\left(1+\frac{1}{n}-\frac{1}{r}\right).
\end{equation}
 If $\psi\in W^{2, \chi}_{\mathrm{loc}}(\Omega)$, with $\chi=\max\{r, 2q-p\}$, then
 $Du\in L^\infty_{\mathrm{loc}}(\Omega)$ and the following estimate
 \begin{eqnarray}\label{stimafinale}
 	\sup_{B_{R/2}}|Du|  \le \left(\frac{C }{R^{2}}\right)^\sigma\left(\int_{B_{ R}}(1+{|Du|})^{p}  \, dx\right)^{\frac{\sigma}{p}}
 \end{eqnarray}
 holds for every ball $B_R\Subset\Omega$, for $\sigma(n, r, p, q)>0$ and with $C=C(||h||_r, ||D^2\psi||_{\chi},\tilde\nu_1, \tilde L_1,p,q,n)$.		
\end{theorem}
Note that the assumption $\psi\in W^{2,\chi}_{\mathrm{loc}}(\Omega)$ reduces to $\psi\in W^{2,r}_{\mathrm{loc}}(\Omega)$  in case $q<n$ (see Subsection \ref{appendix}).
\\
The proof  is achieved through  a Moser type  iteration argument, that allows us to establish an a priori estimate  for the $L^\infty$ norm of the gradient of the solutions. Here,  the presence of the Sobolev coefficients
prevents us to use the a priori assumption $Du\in L^\infty_{\mathrm{loc}}(\Omega)$ to deal with the non standard growth of the functional as a perturbation of the standard one (as done for example in \cite{DFM}).
The main point in our proof is to show that  the Sobolev regularity that we impose on the partial map $x\mapsto {F_z}(x,z)$ yields an higher differentiability for the solutions  and therefore their higher integrability with
an exponent that guarantees the starting of the iteration.
\\
This phenomenon has been widely investigated in case of non constrained problems (see for example  \cite{G15}--\cite{GPdN}, \cite{PdN14-1}--\cite{PdN15}).
\\
We conclude  mentioning that condition \eqref{gap} is sharp also to obtain the Lipschitz continuity of solutions to elliptic equations and systems and minimizers of related functionals with $p,q-$growth (see
\cite{EMM}--\cite{EMM3}) and can be framed into the research concerning regularity results under non standard growth conditions, that, after the pioneering papers by Marcellini \cite{M89}--\cite{M93} has attracted growing
attention, see among the others {\cite{BDM13}, \cite{CM15-1}, \cite{CGGP}, \cite{DFP19}, \cite{HH19}, \cite{M19}, \cite{RT19}.}
\\
\\
The paper is organized as follows: In Section 2 we introduce some notations and collect some results that will be needed in the sequel;  Section 3 is devoted to the proof of Theorem \ref{locbound}; in Section 4  we give the
proof of Theorem \ref{loclip}.
\\

\section{Preliminary results}

\bigskip

\noindent
In this paper we shall denote by $C$ or $c$  a
general positive constant that may vary on different occasions, even within the
same line of estimates.
Relevant dependencies  will be suitably emphasized using
parentheses or subscripts.  In what follows, $B(x,r)=B_r(x)=\{y\in \R^n:\,\, |y-x|<r\}$ will denote the ball centered at $x$ of radius $r$.
We shall omit the dependence on the center and on the radius when no confusion arises.
For a function $u\in L^1(B)$, the symbol
$$u_B := \fint_B u(x)\,dx=\frac{1}{|B|}\int_B u(x)\,dx$$
will denote the integral mean of the function $u$ over the set $B$.

%It is convenient and usual, when dealing with functionals with $p$-growth, to introduce the auxiliary function \begin{equation}\label{aux2}
%V_{p}(\xi ) := \Bigl({\tilde{\mu}^2} +| \xi |^{2} \Bigr)^{\frac{p-2}{4}}\xi ,
%\end{equation}
%defined for all $\xi \in \R^n$ and $\tilde{\mu} \ge 0$.
%For the auxiliary function $V_{p} (\xi)$, we recall the following
%estimate (see the proof of \cite[Lemma 8.3]{Giusti}):
%
%\begin{lemma} \label{Vi}
%Let $1<p<\infty$. There exists a constant $c=c(n,p)>0$
%such that
%$$
%c^{-1}\Bigl({\tilde{\mu}}^2 +| \xi |^2+| \eta |^2 \Bigr)^{\frac{p-2}{2}}\leq
%\frac{|V_{p}(\xi )-V_{p}(\eta )|^2}{|\xi -\eta |^2} \leq
%c\Bigl( {\tilde{\mu}}^2 +|\xi |^2+|\eta |^2 \Bigr)^{\frac{p-2}{2}}
%$$
%for any $\xi$, $\eta \in \R^n$, {${\tilde{\mu}} \ge 0$}.
%\end{lemma}
\medskip
Next Lemma, whose proof can be found in \cite[Lemma 7.1]{Giusti} is  crucial  to establish the local boundedness result.
\medskip
\begin{lemma}\label{ite}
	Let $\alpha>0$ and let $x_i$ be a sequence of positive numbers such that
	\begin{equation}\label{ipoite1}
	x_{i+1}\le C B^i x_i^{1+\alpha}	
	\end{equation}
where $C>0$ and $B>1$ are  constants. If \begin{equation}\label{ipoite2}x_0\le C^{-\frac{1}{\alpha}}B^{-\frac{1}{\alpha^2}},
 \end{equation}	
 then
$$\lim_{i\to +\infty}x_i=0.$$	
\end{lemma}
The following lemma has important application in the so called hole-filling method. Its proof can be found for example in \cite[Lemma 6.1]{Giusti}.
 \begin{lemma}\label{holf} Let $h:[\rho_0, R_{0}]\to \mathbb{R}$ be a non-negative bounded function and $0<\vartheta<1$,
$A, B\ge 0$ and $\beta>0$. Assume that
$$
h(s)\leq \vartheta h(t)+\frac{A}{(t-s)^{\beta}}+B,
$$
for all $\rho_0\leq s<t\leq R_{0}$. Then
$$
h(r)\leq \frac{c A}{(R_{0}-\rho_0)^{\beta}}+cB ,
$$
where $c=c(\vartheta, \beta)>0$.
\end{lemma}
\subsection{Approximation Lemma}
In this subsection we will state a Lemma that will be the main tool in the approximation procedure. For the proof we refer to Lemma 4.1 and Proposition 4.1 in \cite{CupGuiMas}.
\begin{lemma}
\label{apprcupguimas2parte}
Let $F:\Omega\times\R^{n}\to [0,+\infty)$ be a
Carath\'eodory function such that $\tcr{\xi} \mapsto F(x,\xi)$ is $\mathcal{C}^2$ and   there exists $f:\Omega\times [0,+\infty)\mapsto [0,+\infty)$ such that $F(x,\xi)=f(x,|\xi|)$. Assume moreover that $F$
satisfies  assumptions \textnormal{(F1)--(F3)}, with
  $k\in L^r_{\mathrm{loc}}(\Omega)$, for some $r > n$.
 Then there exists a sequence
$({F^{\varepsilon}})$ of Carath\'eodory functions
${F^{\varepsilon}}:\Omega\times\R^{n}\to [0,+\infty)$, monotonically convergent to $F$,
such that
\begin{itemize}
      \item[(I)]for a.e. $x\in \Omega$ , for every  $\xi
      \in
  \R^{n}$, and  for every $\varepsilon_1<\varepsilon_2$, it holds
 $${F^{\varepsilon_1}}(x,\xi)\le {F^{\varepsilon_2}}(x,\xi)\le F(x,\xi),$$
  \item[(II)]  for a.e.  $x\in \Omega$ and every  $\xi\in
   {\R^{n}}$,   we have $$\langle {F_{\xi \xi}^\varepsilon}(x,\xi)\lambda,\lambda\rangle\ge \bar\nu(1+|\xi|^2)^{\frac{p-2}{2}}|\lambda|^2$$
    with
 ${\overline \nu}$ depending only on $p$ and ${\nu_1}$,
 \item[{(III)}] for a.e. $x\in \Omega$ and for every  $\xi
      \in
  \R^{n}$,
 there exist $K_{0},K_{1}$, independent of $\varepsilon$, and
 ${\overline K}_{1}$, depending on $\varepsilon$,
 such
 that
  \begin{equation*}
 \begin{split}
 K_{0}(|\xi|^{p}-1)&\le {F^{\varepsilon}}(x,\xi)\le K_{1}(1+|\xi|)^{q},
 \\
{F^{\varepsilon}}(x,\xi)&\le {\overline K}_{1}(\varepsilon)(1+|\xi|)^{p},
\end{split}
\end{equation*}
 \item[{(IV)}] there exists a constant $C(\varepsilon)>0$ such that
 \begin{equation*}
 \begin{split}
|{F^{\varepsilon}_{x \xi}}(x,\xi)|&\le k(x)(1+|\xi|^2)^{\frac{q-1}{2}},
 \\
|{F^{\varepsilon}_{x \xi}}(x,\xi)|&\le C(\varepsilon){{k^{\varepsilon}}(x)}(1+|\xi|^2)^{\frac{p-1}{2}},
\end{split}
\end{equation*}
for a.e. $x\in \Omega$ and for every  $\xi \in \R^{n}$, {where ${k^{\varepsilon}}$ is a standard mollification of $k$.}
\end{itemize}
In addition, we have
$$F^{\varepsilon}(x,\xi)=f^\varepsilon(x,|\xi|)$$
for every $\xi\in \R^{n}$.
\end{lemma}

\medskip

\subsection{Some useful regularity results}
In this subsection we recall some regularity result that will be needed in the proof of our main results.
The first is an higher differentiability property of the solutions to \eqref{obst-def0} whose proof can be found in \cite{G19}.
\begin{theorem}\label{Chiara}
	Let $u\in \mathcal{K}_\psi(\Omega)$ be a solution to {\eqref{obst-def0}} under the assumptions \textnormal{(F1)--(F3)}. Suppose that there exists $r>n$ such that $h\in L^r_{\mathrm{loc}}(\Omega)$, where $h(x)$ is the
function appearing in \textnormal{(F3)}. Assume moreover that  $2\le p\le q$ are such that {\eqref{gap} is satisfied.}\\
Then the following implication
$$D\psi\in W^{1, 2q-p}_{\mathrm{loc}}(\Omega)\qquad\Longrightarrow \qquad {(1+|Du|^2)^{\frac{p-2}{4}}Du}\in W^{1,2}_{\mathrm{loc}}(\Omega)$$
holds true.
\end{theorem}
Next result, whose proof is contained in \cite{BC19}, concerns the local Lipschitz continuity of solutions to \eqref{obst-def0} under standard growth conditions.
\begin{theorem}\label{CarloMichele}
Let $u\in \mathcal{K}_\psi(\Omega)$ be a solution to {\eqref{obst-def0}} under the assumptions \textnormal{(F1)--(F3)} with $p=q$ and with  $h\in L^\infty_{\mathrm{loc}}(\Omega)$.
If $\psi\in W^{2,\infty}_{\mathrm{loc}}(\Omega)$ then	$Du\in L^\infty_{\mathrm{loc}}(\Omega)$
\end{theorem}

Actually, a careful inspection of the proof of the previous theorem leads us to the following
\begin{theorem}\label{CarloMichelebis}
Let $u\in \mathcal{K}_\psi(\Omega)$ be a solution to {\eqref{obst-def0}} under the assumptions \textnormal{(F1)--(F3)} with $p=q$ and with  $h\in L^\infty_{\mathrm{loc}}(\Omega)$.
If $\psi\in W^{2,r}_{\mathrm{loc}}(\Omega)$, for some $r > n$, then	$Du\in L^\infty_{\mathrm{loc}}(\Omega)$
\end{theorem}
The proof of this result can be easily obtained arguing as in \cite{BC19} by suitably replacing the a priori estimate with that in Corollary \ref{corollario-primo} below and using the linearization technique provided by
Corollary \ref{corollario-secondo} below.

\section{The local boundedness}

\medskip
The aim of this section is to prove that the local boundedness of the obstacle $\psi$ implies the  local boundedness of solution to \eqref{obst-def}. More precisely, we give the

 \begin{proof}[Proof of Theorem \ref{locbound}]
 	
 Let us fix radii $0<\rho<R<R_0$  such that $B_{R_0}\Subset \Omega$ and a cut off function $\eta\in C^{\infty}_0(B_R)$ such that $0\le \eta\le 1$, $\eta=1$ on $B_\rho$, ${|D\eta|}\le \frac{C}{R-\rho}$. By virtue of the
 assumption $\psi\in L^\infty_{\mathrm{loc}}(\Omega)$, we may fix $k\ge \sup_{B_{R_0}}\psi$ and consider $u_k=(u-k)^+=\max\{u-k,0\}$.
Note that
 $$\varphi=u-\eta^\sigma u_k,$$
 with $\sigma\ge\frac{p(q-1)}{p-1}$, is an admissible test function for the variational inequality \eqref{obst-def} {by the properties of $\eta$ and since}
 $$\varphi=\begin{cases}
 	u\ge\psi\qquad\qquad\qquad\qquad\qquad\qquad\qquad\qquad\qquad\qquad \text{if}\,\,u\le k\cr\cr
 	u-\eta^\sigma(u-k)= (1-\eta^\sigma)(u-k)+k\ge k \ge \psi \quad\qquad\text{if}\,\,u\ge k.
 \end{cases}$$
 Using $\varphi$ as test function in \eqref{obst-def}, we get
 \begin{eqnarray*}
 	-\int_{B_R}\big\langle \cA(x,Du),D(\eta^\sigma u_k)\big\rangle \, dx\ge 0
 \end{eqnarray*}
 and so, exploiting the calculations and recalling the definition of $u_k$,
 \begin{eqnarray}\label{est1}
 	\int_{B_R\cap\{u\ge k\}}\eta^\sigma\big\langle \cA(x,Du),Du\,\big\rangle dx\le \int_{B_R\cap\{u\ge k\}}\big\langle \cA(x,Du),-\sigma(u-k)\eta^{\sigma-1}D\eta\big\rangle\, dx.
 \end{eqnarray}
 The assumption $(\cA 2)$ obviously implies that $\langle \mathcal{A}(x, \xi_1)-\mathcal{A}(x, \xi_2),\xi_1-\xi_2\big\rangle>0$ and so
 $$ \big\langle \mathcal{A}(x, \xi_1),\xi_2\big\rangle  \le \big\langle \mathcal{A}(x, \xi_1), \xi_1\big\rangle-\big\langle \mathcal{A}(x, \xi_2), \xi_1\big\rangle+\big\langle \mathcal{A}(x, \xi_2), \xi_2\big\rangle$$
 that we use with $\xi_1=Du$ and $\xi_2=-2\sigma\frac{D\eta}{\eta}(u-k)$ in the set $\{\eta\not=0\}$, thus obtaining
 \begin{eqnarray}\label{est2}
&&\big\langle \cA(x,Du),-\sigma(u-k)\eta^{\sigma-1}D\eta\big\rangle = \frac{\eta^\sigma}{2}\left\langle \cA(x,Du),-2\sigma(u-k)\eta^{-1}D\eta\right\rangle\cr\cr
&\le& \frac{\eta^\sigma}{2}\Bigg(\big\langle \cA(x,Du),Du\,\big\rangle-\big\langle \cA(x,-2\sigma(u-k)\eta^{-1}D\eta),Du\,\big\rangle\cr\cr
&&\qquad+\big\langle \cA(x,-2\sigma(u-k)\eta^{-1}D\eta),-2\sigma(u-k)\eta^{-1}D\eta\big\rangle\Bigg)
 \end{eqnarray}
 Using \eqref{est2} to estimate  the right hand side of \eqref{est1}, we get
\begin{eqnarray}\label{est3}
 	&&\int_{B_R\cap\{u\ge k\}}\eta^\sigma\big\langle \cA(x,Du),Du\,\big\rangle dx\le \int_{B_R\cap\{u\ge k\}}\frac{\eta^\sigma}{2}\big\langle \cA(x,Du),Du\,\big\rangle\,dx\cr\cr
 	&&\qquad-\int_{B_R\cap\{u\ge k\}}\frac{\eta^\sigma}{2}\big\langle \cA(x,-2\sigma(u-k)\eta^{-1}D\eta),Du\,\big\rangle\,dx\cr\cr
&&\qquad+\int_{B_R\cap\{u\ge k\}}\frac{\eta^\sigma}{2}\big\langle \cA(x,-2\sigma(u-k)\eta^{-1}D\eta),-2\sigma(u-k)\eta^{-1}D\eta\big\rangle\,dx.
 \end{eqnarray}
 Reabsorbing the first integral in the right hand side by the left hand side, we get
 \begin{eqnarray}\label{est4}
 	&&\frac{1}{2}\int_{B_R\cap\{u\ge k\}}\eta^\sigma\big\langle \cA(x,Du),Du\,\big\rangle dx\cr\cr
 	&\le &-\int_{B_R\cap\{u\ge k\}}\frac{\eta^\sigma}{2}\big\langle \cA(x,-2\sigma(u-k)\eta^{-1}D\eta),Du\,\big\rangle\,dx\cr\cr
&&\quad+\int_{B_R\cap\{u\ge k\}}\frac{\eta^\sigma}{2}\big\langle \cA(x,-2\sigma(u-k)\eta^{-1}D\eta),-2\sigma(u-k)\eta^{-1}D\eta\big\rangle\,dx
 \end{eqnarray}
 and so
  \begin{eqnarray}\label{est5}
 	&&\int_{B_R\cap\{u\ge k\}}\eta^\sigma\big\langle \cA(x,Du),Du\,\big\rangle dx\cr\cr
 	&\le &\int_{B_R\cap\{u\ge k\}}\eta^\sigma|\cA(x,-2\sigma(u-k)\eta^{-1}D\eta)||Du|\,dx\cr\cr
&&\quad+\int_{B_R\cap\{u\ge k\}}\eta^\sigma\big| \cA(x,-2\sigma(u-k)\eta^{-1}D\eta)||2\sigma(u-k)\eta^{-1}D\eta\big|\,dx .
 \end{eqnarray}
 Using $(\cA 2)$ in the left hand side and $(\cA 1)$ in the right hand side of the previous estimate, we deduce
 \begin{eqnarray}\label{est6}
 	&&\nu\int_{B_R\cap\{u\ge k\}}\eta^\sigma|Du|^p\,dx\cr\cr
 	&\le & c(q,\sigma,\ell)\int_{B_R\cap\{u\ge k\}}\eta^\sigma\big(1+|u-k|^{2}|\eta^{-1}D\eta|^2)^{\frac{q-1}{2}}|Du|\,dx\cr\cr
&&\quad+c(q,\sigma,\ell)\int_{B_R\cap\{u\ge k\}}\eta^\sigma\big(1+|u-k|^{2}\,|\eta^{-1}D\eta|^2)^{\frac{q-1}{2}}|(u-k)\eta^{-1}D\eta|\,dx.
 \end{eqnarray}
By virtue of Young's inequality, we get
\begin{eqnarray}\label{est7}
 	&&\nu\int_{B_R\cap\{u\ge k\}}\eta^\sigma|Du|^p\,dx\le \frac{\nu}{2}\int_{B_R\cap\{u\ge k\}}\eta^\sigma|Du|^p\,dx\cr\cr
 	&&\qquad+ c(q,\sigma,\ell,\nu)\int_{B_R\cap\{u\ge k\}}\eta^\sigma\big(1+|u-k|^{2}\,|\eta^{-1}D\eta|^2)^{\frac{p(q-1)}{2(p-1)}}\,dx\cr\cr
&&\quad+c(q,\sigma,\ell)\int_{B_R\cap\{u\ge k\}}\eta^\sigma\big(1+|u-k|^{2}\,|\eta^{-1}D\eta|^2)^{\frac{q}{2}}\,dx
 \end{eqnarray}
Reabsorbing the first integral in the right hand side by the left hand side, we get
\begin{eqnarray}\label{est8}
 	&&\int_{B_R\cap\{u\ge k\}}\eta^\sigma|Du|^p\,dx\le  c(q,\sigma,\ell,\nu)\int_{B_R\cap\{u\ge k\}}\eta^\sigma\big(1+|u-k|^{2}\,|\eta^{-1}D\eta|^2)^{\frac{p(q-1)}{2(p-1)}}\,dx\cr\cr
 	&\le& c(q,\sigma,\ell,\nu)\int_{B_R\cap\{u\ge k\}}\eta^\sigma\big(1+|u-k|^{2}\,|\eta^{-1}D\eta|^2)^{\frac{p(q-1)}{2(p-1)}}\,dx\cr\cr
 	&\le& c(q,\sigma,\ell,\nu)\int_{B_R\cap\{u\ge k\}}\eta^{\sigma-\frac{p(q-1)}{(p-1)}}|D\eta|^{\frac{p(q-1)}{(p-1)}}|u-k|^{\frac{p(q-1)}{(p-1)}}\,dx\cr\cr
 	&&\qquad+c(q,\sigma,\ell,\nu)|B_R\cap\{u\ge k\}|,
 	 \end{eqnarray}
 	 where we used that $q\le \frac{p(q-1)}{p-1}$, since $p\le q$ and that $\sigma\ge\frac{p(q-1)}{(p-1)}$. By the assumption $q<p\frac{n-1}{n-p}$ and  the properties of $\eta$, we get
 \begin{eqnarray}\label{est9}
 	&&\int_{B_R\cap\{u\ge k\}}\eta^\sigma|Du|^p\,dx\cr\cr
 	 	&\le& \frac{c(q,\sigma,\ell,\nu)}{(R-\rho)^{\frac{p(q-1)}{(p-1)}}}\left(\int_{B_R\cap\{u\ge k\}}\eta^{ p^* \left( \frac{\sigma (p-1)}{p(q-1)}-1\right)}|u-k|^{p^*}\,dx\right)^{\frac{p(q-1)}{p^*(p-1)}}|B_R\cap\{u\ge
 k\}|^{1-\frac{p(q-1)}{p^*(p-1)}}\cr\cr
 	&&\qquad+c(q,\sigma,\ell,\nu)|B_R\cap\{u\ge k\}|.
 	 \end{eqnarray}	
 	 On the other hand, the Sobolev imbedding Theorem implies
 	 \begin{eqnarray*}
 	 	\left(\int_{B_R}\big|\eta^{\frac{\sigma}{p}}(u-k)^+|^{p^*}\,dx\right)^{\frac{p}{p^*}}\le c \int_{B_R}\eta^{\sigma-p}|D\eta|^p\big|(u-k)^+|^{p}\,dx+c\int_{B_R\cap\{u\ge k\}}\eta^{\sigma}\big|Du|^{p}\,dx,
 	 \end{eqnarray*}
 	 where as usual $$p^*=\begin{cases}
 	 	\frac{np}{n-p}\qquad\qquad\qquad\qquad\text{if}\,\, p<n\cr\cr
 	 	\text{any finte exponent}\qquad \text{if}\,\, p\ge n.
 	 \end{cases}$$
 Hence, using the properties of $\eta$  and  \eqref{est9} in previous estimate, we get
 \begin{eqnarray*}
 	 &&	\left(\int_{B_\rho\cap\{u\ge k\}}(u-k)^{p^*}\,dx\right)^{\frac{p}{p^*}}\le \frac{c}{(R-\rho)^p} \int_{B_R\cap\{u\ge k\}}\eta^{\sigma-p}(u-k)^{p}\,dx\cr\cr
 	 	&+& \frac{c(q,\sigma,\ell,\nu)}{(R-\rho)^{\frac{p(q-1)}{(p-1)}}}\left(\int_{B_R \cap\{u\ge k\}}\eta^{ p^* \left( \frac{\sigma(p-1)}{p(q-1)}-1\right)}|u-k|^{p^*}\,dx\right)^{\frac{p(q-1)}{p^*(p-1)}}|B_R\cap\{u\ge
 k\}|^{1-\frac{p(q-1)}{p^*(p-1)}} \cr\cr
 	&&\qquad+c(q,\sigma,\ell,\nu)|B_R\cap\{u\ge k\}|\cr\cr
 	&\le& \frac{c|B_R\cap\{u\ge k\}|^{\frac{p}{n}}}{(R-\rho)^p}\left( \int_{B_R\cap\{u\ge k\}}(u-k)^{p^*}\,dx\right)^{\frac{p}{p^*}}\cr\cr
 	 	&+& \frac{c(q,\sigma,\ell,\nu)}{(R-\rho)^{\frac{p(q-1)}{(p-1)}}}\left(\int_{B_R\cap\{u\ge k\}}(u-k)^{p^*}\,dx\right)^{\frac{p(q-1)}{p^*(p-1)}}|B_R\cap\{u\ge k\}|^{1-\frac{p(q-1)}{p^*(p-1)}}\cr\cr
 	&&\qquad+c(q,\sigma,\ell,\nu)|B_R\cap\{u\ge k\}|,
 	 \end{eqnarray*}
 	 where we used that $\sigma\ge\frac{p(q-1)}{p-1}>p$.
 	 From previous estimate we deduce that
 	 \begin{eqnarray}\label{est10}
 	 &&	\int_{B_\rho\cap\{u\ge k\}}(u-k)^{p^*}\,dx \le \frac{c|B_R\cap\{u\ge k\}|^{\frac{p^*}{n}}}{(R-\rho)^{p^*}} \int_{B_R\cap\{u\ge k\}}(u-k)^{p^*}\,dx\cr\cr
 	 	&+& \frac{c(q,\sigma,\ell,\nu)}{(R-\rho)^{\frac{p^*(q-1)}{(p-1)}}}\left(\int_{B_R\cap\{u\ge k\}}(u-k)^{p^*}\,dx\right)^{\frac{q-1}{p-1}}|B_R\cap\{u\ge k\}|^{\frac{p^*}{p}-\frac{q-1}{p-1}}\cr\cr
 	&&\qquad+c(q,\sigma,\ell,\nu)|B_R\cap\{u\ge k\}|^{\frac{p^*}{p}}.
 	 \end{eqnarray}
 	 Note that for $h<k$ we have $B_R\cap\{u\ge k\}\subset B_R\cap\{u\ge h\}$ and
 	 \begin{eqnarray*}
 	 &&\int_{B_R\cap\{u\ge h\}}(u-h)^{p^*}\,dx\ge \int_{B_R\cap\{u\ge k\}}(u-h)^{p^*}\,dx\cr\cr
 	 &\ge& \int_{B_R\cap\{u\ge k\}}(k-h)^{p^*}\,dx=(k-h)^{p^*}|B_R\cap\{u\ge k\}|	
 	 \end{eqnarray*}
 	 and so
 	 $$|B_R\cap\{u\ge k\}|\le \frac{1}{(k-h)^{p^*}}\int_{B_R\cap\{u\ge h\}}(u-h)^{p^*}\,dx.$$
 	 Moreover
 	 $$\int_{B_R\cap\{u\ge k\}}(u-k)^{p^*}\,dx\le \int_{B_R\cap\{u\ge k\}}(u-h)^{p^*}\,dx\le \int_{B_R\cap\{u\ge h\}}(u-h)^{p^*}\,dx.$$
 	 Inserting
 	 previous estimates in \eqref{est10}, we obtain
 	\begin{eqnarray}\label{est11}
 	 &&	\int_{B_\rho\cap\{u\ge k\}}(u-k)^{p^*}\,dx\le  \frac{c}{(R-\rho)^{p^*}(k-h)^{\frac{(p^*)^2}{n}}}\left( \int_{B_R\cap\{u\ge h\}}(u-h)^{p^*}\,dx\right)^{\frac{p^*}{p}}\cr\cr
 	 	&&\qquad+ \frac{c(q,\sigma,\ell,\nu)}{(R-\rho)^{\frac{p^*(q-1)}{(p-1)}}(k-h)^{p^*\left(\frac{p^*}{p}-\frac{q-1}{p-1}\right)}}\left(\int_{B_R\cap\{u\ge h\}}(u-h)^{p^*}\,dx\right)^{\frac{p^*}{p}}\cr\cr
 	&&\qquad+\frac{c(q,\sigma,\ell,\nu)}{(k-h)^{\frac{(p^*)^2}{p}}}\left(\int_{B_R\cap\{u\ge h\}}(u-h)^{p^*}\,dx\right)^{\frac{p^*}{p}},
 	 \end{eqnarray}
 	 for every $\rho<R<R_0$ and every $0<h<k$. Define now two  sequences by setting
 	 $$\rho_i=\frac{R_0}{2}\left(1+\frac{1}{2^i}\right)$$
 	 $$k_i=2d\left(1-\frac{1}{2^{i+1}}\right),$$
 	 where $d\ge\sup_{B_{R_0}} \psi $ will be determined later.
 	 Estimate \eqref{est11} can be written as
 	 \begin{eqnarray}\label{est12}
 	 &&	\int_{B_{\rho_{i+1}}\cap\{u\ge k_{i+1}\}}(u-k_{i+1})^{p^*}\,dx\le  \frac{c}{(\rho_{i}-\rho_{i+1})^{p^*}(k_{i+1}-k_i)^{\frac{(p^*)^2}{n}}}\left( \int_{B_{\rho_i}\cap\{u\ge
 k_i\}}(u-k_i)^{p^*}\,dx\right)^{\frac{p^*}{p}}\cr\cr
 	 	&&\qquad+ \frac{c(q,\sigma,\ell,\nu)}{(\rho_{i}-\rho_{i+1})^{\frac{p^*(q-1)}{(p-1)}}(k_{i+1}-k_i)^{p^*\left(\frac{p^*}{p}-\frac{q-1}{p-1}\right)}}\left(\int_{B_{\rho_i}\cap\{u\ge
 k_i\}}(u-k_i)^{p^*}\,dx\right)^{\frac{p^*}{p}}\cr\cr
 	&&\qquad+\frac{c(q,\sigma,\ell,\nu)}{(k_{i+1}-k_i)^{\frac{(p^*)^2}{p}}}\left(\int_{B_{\rho_i}\cap\{u\ge k_i\}}(u-k_i)^{p^*}\,dx\right)^{\frac{p^*}{p}}\cr\cr
 	&=& \frac{c(R_0)}{d^\frac{(p^*)^2}{n}}2^{(i+2)p^*+(i+2){\frac{(p^*)^2}{n}}}\left( \int_{B_{\rho_i}\cap\{u\ge k_i\}}(u-k_i)^{p^*}\,dx\right)^{\frac{p^*}{p}}\cr\cr
 	 	&&\qquad+\left(\frac{c(R_0)}{d^{p^*\left(\frac{p^*}{p}-\frac{q-1}{p-1}\right)}}+\frac{c(R_0)}{d\frac{(p^*)^2}{p}}\right)2^{(i+2){\frac{(p^*)^2}{p}}}\left(\int_{B_{\rho_i}\cap\{u\ge
 k_i\}}(u-k_i)^{p^*}\,dx\right)^{\frac{p^*}{p}}\cr\cr
 	 	&\le& \frac{c(R_0)}{d^\vartheta}2^{(i+2){\frac{(p^*)^2}{p}}}\left(\int_{B_{\rho_i}\cap\{u\ge k_i\}}(u-k_i)^{p^*}\,dx\right)^{\frac{p^*}{p}},
 	 \end{eqnarray}
 	 where $$\frac{1}{d^\vartheta}=\max\left\{\frac{1}{d^\frac{(p^*)^2}{n}},\frac{1}{d^{p^*\left(\frac{p^*}{p}-\frac{q-1}{p-1}\right)}},\frac{1}{d^\frac{(p^*)^2}{p}}\right\}.$$
 	 Setting $$\Phi_i=\int_{B_{\rho_i}\cap\{u\ge k_i\}}(u-k_i)^{p^*}\,dx$$
 	 estimate \eqref{est12} can be written as follows
 	 $$\Phi_{i+1}\le \frac{C}{d^\vartheta} \left(2^{\frac{(p^*)^2}{p}}\right)^i\Phi_{i}^{\frac{p^*}{p}},$$
 	 i.e. the sequence $\Phi_i$ satisfies assumption \eqref{ipoite1}  with $\alpha=\frac{p}{n-p}$.
 	 In order to have also assumption \eqref{ipoite2} satisfied it suffices to choose $$d\ge \tilde C \left(\int_{B_{R_0}}((u-\sup_{B_{R_0}}\psi)^+)^{p^*}\right)^{\frac{p}{\vartheta(n-p)}},$$
 	 for  a suitable $\tilde C$ depending on $n,p,q,\sigma, \ell,\nu$. Therefore by Lemma \ref{ite} we have 	 $$\lim_i\Phi_i=0$$
 	 and so, by the definition of $k_i$ and $\rho_i$, we conclude that
 	 $$|B_{R_0/2}\cap\{ u\ge 2d\}|=0,$$
 	 i.e.
 	 $$\sup_{B_{R_0/2}}u\le 2d.$$
 	 As it is customary when dealing with the local boundedness of the minimizers, {(\cite{Giusti}),} the conclusion follows by changing $u$ in $-u$ and arguing in a similar way to deduce that $$\inf_{B_{R_0/2}}u\ge 2d.$$
% 	 that yields
% \begin{eqnarray}\label{est12}
% 	 &&	\int_{B_\rho\cap\{u\ge k\}}(u-k)^{p^*}\,dx\cr\cr
% 	&\le& \frac{c}{(R-\rho)^{p^*}(k-h)^{\frac{p}{n}}}\left( \int_{B_R\cap\{u\ge h\}}(u-h)^{p^*}\,dx\right)^{\frac{n}{n-p}}\cr\cr
% 	 	&+& \frac{c(q,\sigma,\ell,\nu)}{(R-\rho)^{\frac{p(q-1)}{(p-1)}}}\left(\int_{B_R\cap\{u\ge h\}}(u-h)^{p^*}\,dx\right)^{\frac{n}{n-p}}\left(\frac{1}{(k-h)}\right)^{\frac{np^*}{n-p}\frac{p(q-1)-n(q-p)}{n(p-1)}}\cr\cr
% 	&&\qquad+\frac{c(q,\sigma,\ell,\nu)}{(k-h)^{\frac{np^*}{n-p}}}\left(\int_{B_R\cap\{u\ge h\}}(u-h)^{p^*}\right)^{\frac{n}{n-p}}
% 	 \end{eqnarray} 	
 \end{proof}

\section{The Lipschitz continuity}

This section is devoted to the proof of Theorem \ref{loclip}.  The first step consists in the linearization argument, i.e. we shall   show that solutions of \eqref{obst-def0} solve a suitable elliptic equation. Then, we prove
the a priori estimate for the $L^\infty$ norm of the gradient of the solutions and finally third one is the approximation procedure. For the sake of clarity, we will insert the linearization procedure in a subsection.

	\subsection{Linearization}
	
	\label{linear}
The first step in the proof of our main result is the following

\begin{theorem}
	Let $u\in \mathcal{K}_\psi(\Omega)$ be a solution to {\eqref{obst-def0}} under the assumptions \textnormal{(F1)--(F3)}.
	Suppose that there exists $r>n$ such that $h\in L^r_{\mathrm{loc}}(\Omega)$, where $h(x)$ is the function appearing in \textnormal{(F3)}. Assume moreover that $2\le p\le q$ satisfy \eqref{gap}.
	If $\psi\in W^{2, \chi}_{\mathrm{loc}}(\Omega)$, with $\chi=\max\{r, 2q-p\},$	then there exists $g\in L^r_{\mathrm{loc}}(\Omega)$ such that
	\begin{equation*}
		\mathrm{div}\mathcal{A}(x,Du)= {g}\qquad \text{in}\,\,\Omega.
	\end{equation*}
\end{theorem}

\begin{proof} For $\varepsilon>0$, let us consider a smooth function $\kappa_{\varepsilon} : (0, \infty) \to [0, 1]$ such that $\kappa'_\varepsilon(s) \le 0 $ for all $s\in (0, \infty)$ and
\[
\kappa_\varepsilon(s) =
     \begin{cases}
       1 &\quad\text{for} \quad s \le \varepsilon \\
       0 &\quad\text{for} \quad s \ge 2\varepsilon. \\
     \end{cases}
\]
Under our assumptions, by virtue of Theorem \ref{Chiara} {and by Sobolev embedding, we have $Du\in L^{\frac{np}{n-2}}_{\mathrm{loc}}(\Omega)\hookrightarrow L^q_{\mathrm{loc}}(\Omega)$,} therefore we can use the results
contained in Subsection \ref{EL}
to conclude that
\begin{equation*}
    \varphi = u+t\cdot \eta \cdot  \kappa_{\varepsilon}(u-\psi)
\end{equation*}
{can be used as a test function in the variational inequality} \eqref{obst-def}, {with $\eta \in C_{0}^1 (\Omega) $, $\eta \ge 0$ and $0 < t <<1$},  {so that}
\begin{equation*}
     \int_{\Omega} \Big\langle \mathcal{A}(x, Du), D(\eta \kappa_{\varepsilon}(u-\psi)) \Big\rangle \, dx \ge 0  \qquad \forall  \, \eta \in C_{0}^1(\Omega) \, .
\end{equation*}
On the other hand
\[
\eta \mapsto L(\eta)=\int_{\Omega} \Big\langle\mathcal{A}(x, Du) , D(\eta \kappa_{\varepsilon}(u-\psi)) \Big\rangle\, dx
\]
is a bounded positive linear functional, thus by the Riesz representation theorem there exists a nonnegative measure $\lambda_{\varepsilon}$ such that
\begin{equation*}
     \int_{\Omega} \Big\langle\mathcal{A}(x, Du) ,  D(\eta \kappa_{\varepsilon}(u-\psi)) \Big\rangle\, dx = \int_{\Omega} \eta d\lambda_{\varepsilon}  \qquad \forall  \, \eta \in C_{0}^1(\Omega) \, .
\end{equation*}
Arguing as in \cite{BC19}, it is possible to show that the measure $\lambda_{\varepsilon}$ is independent from $\varepsilon$, therefore we can rewrite our representation equation without the $\varepsilon$ dependence on the
measure
\begin{equation}\label{prima}
    \int_{\Omega} \Big\langle\mathcal{A}(x, Du) , D(\eta \kappa_{\varepsilon}(u-\psi)) \Big\rangle \, dx = \int_{\Omega} \eta \, d\lambda  \qquad \forall  \, \eta \in C_{0}^1(\Omega) \, .
\end{equation}
By virtue of the assumption $\psi\in W^{2, \chi}_{\mathrm{loc}}(\Omega)$  with $\chi\ge 2q-p$  and \eqref{gap},
%since $r>2q-p$ {(see Subsection \ref{appendix})},
{we are legitimate to apply Theorem  \ref{Chiara}} thus getting
\[
V_p(Du) := (1 + |Du|^2)^{\frac{p-2}{4}} Du \in W^{1,2}_{\rm loc}(\Omega).
\]
Therefore, we can integrate by parts the left hand side of \eqref{prima}, and get
\begin{equation*}
    \int_{\Omega} -\text{div}(\mathcal{A}(x, Du) ) \eta \kappa_{\varepsilon}(u-\psi) \, dx = \int_{\Omega} \eta \, d\lambda \qquad \forall  \, \eta \in C_{0}^1(\Omega) \, ,
\end{equation*}
where we used also that $\xi\mapsto \mathcal{A}(x,\xi)\in C^1$ and $x\mapsto \mathcal{A}(x,\xi)\in W^{1,r}$.
With the purpose to identify the measure $\lambda$, we pass to the limit  as $\varepsilon \searrow 0 $
\begin{equation}
\label{eq-g}
    \int_{\Omega} -\text{div}(\mathcal{A}(x, Du) ) \chi_{ \left[u = \psi \right] } \eta \, dx = \int_{\Omega} \eta \, d\lambda \qquad \forall  \, \eta \in C_{0}^1(\Omega) \, .
\end{equation}
 Let us set
 \begin{equation}
 \label{def-g}
 g:={\text{div}(\mathcal{A}(x, Du)) \chi_{ \left[u = \psi \right] };}
 \end{equation}
this entails
\begin{equation*}
   -  \int_{\Omega} g \eta \, dx = \int_{\Omega} \eta \, d\lambda \qquad \forall  \, \eta \in C_{0}^1(\Omega) \, .
\end{equation*}
We remark that
\begin{equation*}
    { \int_{\Omega}} \Big\langle\mathcal{A}(x, Du), D(\eta (1-\kappa_{\varepsilon})(u-\psi)) \Big\rangle \, dx = 0  \qquad \forall  \, \eta \in C_{0}^1(\Omega) \,
\end{equation*}
since $(1-\kappa_{\varepsilon})(s)$ has support $[\varepsilon, +\infty)$ and so combining the previous equality with \eqref{prima}, we get
\begin{equation}
\label{eq:g}
    \int_{\Omega} \mathcal{A}(x, Du) \cdot D\eta \, dx = \int_{\Omega} g \eta \, dx \qquad \forall  \, \eta \in C_{0}^1(\Omega) \, .
\end{equation}
We are left to obtain an $L^{r}$ estimate for $g$: since $Du=D\psi$ a.e. on the contact set and it is zero elsewhere, by (F2) and (F3),  we have
\begin{equation*}
\begin{split}
    |g| &= \left|\text{div}(\mathcal{A}(x, Du) ) \chi_{ \left[u = \psi \right] } \right| \le   \left|\text{div}(\mathcal{A}(x, D\psi) ) \right| \\
    &\le \sum_{k=1}^n | F_{\xi_k x_k}(x, D\psi) | + \sum_{k,i=1}^n | F_{\xi_k \xi_i}(x, D\psi) \psi_{x_k x_i} | \\
    &\le h(x)(1+|D\psi|^2)^{\frac{q-1}{2}}+ {\tilde L_1}(1+|D\psi|^2)^{\frac{q-2}{2}}|D^2 \psi| \,
\end{split}
\end{equation*}
The assumption $D\psi \in W^{1, \chi}_{\rm loc}(\Omega; \mathbb{R}^n)$ with $\chi >n$ implies, through the Sobolev imbedding Theorem that $D\psi\in L^\infty_{\mathrm{loc}}(\Omega)$, and so it is immediate to deduce that $g\in
L^r_{\mathrm{loc}}(\Omega)$.
\end{proof}

{In the case $p = q$ we easily get the following
\begin{corollary}
\label{corollario-primo}
	Let $u\in \mathcal{K}_\psi(\Omega)$ be a solution to {\eqref{obst-def0}} under the assumptions \textnormal{(F1)--(F3)} with $p=q$ and with  $h\in L^\infty_{\mathrm{loc}}(\Omega)$.
	If $\psi\in W^{2,r}_{\mathrm{loc}}(\Omega)$, for some $r > n$, then there exists $g\in L^r_{\mathrm{loc}}(\Omega)$ such that
	\begin{equation*}
		\mathrm{div}\mathcal{A}(x,Du)= g\qquad \text{in}\,\,\Omega.
	\end{equation*}
\end{corollary}
}

\subsection{Proof of Theorem \ref{loclip}}
Now, we are ready to give the

\begin{proof}[Proof of Theorem \ref{loclip}] We separate the a priori estimate and the approximation argument.

\medskip
{\bf Step 1. The a priori estimate.}
{Let us fix a ball $B_{R_0}\Subset \Omega$ and radii $\frac{R_{0}}{2}<\bar\rho<\rho<t_1<t_2<R<\bar R<R_{0}$ that will be needed in the three iteration procedures, constituting the essential steps in our proof.
Let us start with \eqref{eq:g}. Recall that by Theorem \ref{Chiara} we have $$(1 + |Du|^2)^{\frac{p-2}{2}} |D^2 u|^2 \in L^1_{\rm loc}(\Omega)$$
and we are assuming  a priori that
\begin{equation}
\label{regE}
u \in   W^{1, \infty}_{\rm loc}(\Omega),
\end{equation}
so that the following system
\begin{equation}
\label{second_variation}
\int_{\Omega} \left (\sum_{i,j=1}^n  F_{\xi_i \xi_j}(x, Du)u_{x_j x_s}D_{x_i} \varphi  \, dx + \sum_{i=1}^n F_{\xi_i x_s}(x, Du) D_{x_i} \varphi \right ) = \int_{\Omega} g \, D_{x_s} \varphi \, dx,
\end{equation}
holds for all $s = 1, \dots, n$ and for all $\varphi \in W^{1,p}_0(\Omega)$; here $g$ is the function which has been introduced in \eqref{def-g}. We  choose $\eta \in \mathcal{C}^1_0(\Omega)$ such that $0 \le \eta \le 1$,
$\eta \equiv 1$ on $B_{t_1}$, $\eta \equiv 0$ outside $B_{t_2}$ and {$|D \eta| \le \frac{C}{(t_2-t_1)}$}. The a priori assumption \eqref{regE} allows us to test \eqref{second_variation} with $\varphi=\eta^2
(1+|Du|^2)^{\gamma} u_{x_s}$, for some $\gamma \ge 0$ so that
$$ D_{x_i} \varphi =2 \eta \eta_{x_i} (1 +|Du|^2)^{\gamma} u_{x_s} + 2 \eta^2 \gamma (1 +|Du|^2)^{\gamma-1} |Du| D_{x_i}(|Du|) u_{x_s} + \eta^2 (1 +|Du|^2)^{\gamma} u_{x_s x_i}.$$
Inserting in \eqref{second_variation} we get:
\begin{eqnarray*}
0 &=& \int_{\Omega} \sum_{i,j=1}^n F_{\xi_i \xi_j}(x, Du) u_{x_j x_s} 2\eta \eta_{x_i} (1 +|Du|^2)^{\gamma} u_{x_s} \, dx\\
&& +  \int_{\Omega} \sum_{i,j=1}^n F_{\xi_i \xi_j}(x, Du) u_{x_j x_s} \eta^2 (1 +|Du|^2)^{\gamma} u_{x_s x_i} \, dx\\
&& +  \int_{\Omega} \sum_{i,j=1}^n F_{\xi_i \xi_j}(x, Du) u_{x_j x_s} 2\eta^2 \gamma  (1 +|Du|^2)^{\gamma-1} |Du| D_{x_i}(|Du|) u_{x_s} \, dx\\
&& + \int_{\Omega} \sum_{i=1}^n F_{\xi_i x_s}(x, Du) 2\eta \eta_{x_i}  (1 +|Du|^2)^{\gamma} u_{x_s} \, dx\\
&& +  \int_{\Omega} \sum_{i=1}^n F_{\xi_i x_s }(x, Du) \eta^2 (1 + |Du|^2)^{\gamma}  u_{x_s x_i} \, dx\\
&& +  \int_{\Omega} \sum_{i=1}^n F_{\xi_i x_s}(x, Du) 2\eta^2 \gamma  (1 + |Du|^2)^{\gamma-1}  |Du| D_{x_i}(|Du|) u_{x_s} \, dx \\
&& - \int_{\Omega} g  2 \eta  \eta_{x_s} \,  (1 + |Du|^2)^{\gamma}  \, u_{x_s} \, dx\\
&& - \int_{\Omega} g  2 \eta^2 \gamma \,  (1 + |Du|^2)^{\gamma - 1}  |Du| \, D_{x_s}(|Du|) \, u_{x_s} \, dx \\
&& - \int_{\Omega} g \, \eta^2  (1 + |Du|^2)^{\gamma} u_{x_s x_s} \, dx\\
&& =: I_{1,s} + I_{2,s} + I_{3,s} + I_{4,s} + I_{5,s} + I_{6,s} + I_{7,s} + I_{8,s} + I_{9,s}.
\end{eqnarray*}
Let us sum in the previous  equation  all terms with respect to $s$ from 1 to $n$, and we denote by $I_1-I_9$ the corresponding integrals.
\\
Previous equality yields
\begin{equation}\label{start}
	I_2+I_3\le |{I}_1|+|{I}_4|+|{I}_5|+|{I}_6|+|{I}_7|+|{I}_8|+|{I}_9|.
\end{equation}
Using the left inequality in assumption (F2) and the fact that $D_{x_j}(|Du|)|Du|=\sum_{k=1}^n u_{x_j x_k} u_{x_k}$, we can estimate the term $I_3$ as follows:
\begin{eqnarray*}\label{I3}
{I}_3 &=& \int_{\Omega}  \sum_{i,j,s =1}^n F_{\xi_i \xi_j}(x, Du) u_{x_j x_s} \left [2\eta^2 \gamma  (1 + |Du|^2)^{\gamma-1}  D_{x_i}(|Du|)|Du| \right ] u_{x_s} \, dx\cr\cr
%&\ge& 2 \gamma \, \int_{\Omega} \eta^2 |Du|^{2\gamma-1} \sum_{i,j,s =1}^n f_{\xi_i \xi_j}(x, Du) D_{x_i}(|Du|) u_{x_j x_s} u_{x_s} \, dx\\
%&=& 2 \gamma \, \int_{\Omega} \eta^2 |Du|^{2\gamma-1} \sum_{i,j =1}^n f_{\xi_i \xi_j}(x, Du) D_{x_i}(|Du|) \left(\sum_{s=1}^n u_{x_j x_s} u_{x_s} \right)\, dx\\
&\ge& 2 \gamma \, \int_{\Omega} \eta^2 (1+|Du|^2)^{\gamma-1}|Du| \sum_{i,j =1}^n F_{\xi_i \xi_j}(x, Du) D_{x_i}(|Du|) D_{x_j}(|Du|) \, dx\cr\cr
&\ge& 2 \gamma \, \int_{\Omega} \eta^2 (1+|Du|^2)^{\gamma-1}|Du|  |D(|Du|)|^2  (1 + |Du|^2)^{\frac{p-2}{2}}  \, dx \ge 0.
\end{eqnarray*}
Therefore, estimate \eqref{start} implies
\begin{eqnarray}\label{ristart}
	I_2\le |{I}_1|+|{I}_4|+|{I}_5|+|{I}_6|+|{I}_7|+|{I}_8|+|{I}_9|.
\end{eqnarray}

By the Cauchy-Schwarz inequality, the Young inequality and the right inequality in assumption  (F2), we have
\begin{eqnarray}\label{I1}
|{I}_1| &=& 2\left |\int_{\Omega}  \eta  (1 + |Du|^2)^{\gamma}   \sum_{i,j,s =1}^n F_{\xi_i \xi_j}(x, Du) u_{x_j x_s} \eta_{x_i} u_{x_s} \, dx\right |\cr\cr
%& = &
%2\Bigg| \int_{\Omega}  \eta  (1 + |Du|^2)^{\gamma} \cr\cr
%&& \times \sum_{s=1}^n \left( \sum_{i,j=1}^n F_{\xi_i \xi_j}(x, Du) \eta_{x_i} \eta_{x_j} u_{x_s}^2 \right)^{1/2} \left(\sum_{i,j=1}^n
%F_{\xi_i \xi_j}(x, Du) u_{x_s x_i} \, u_{x_s x_j} \right)^{1/2} \, dx \Bigg| \cr\cr
&\le& 2\int_{\Omega}  \eta  (1 + |Du|^2)^{\gamma} \cr\cr
&& \times \left \{ \sum_{i,j,s =1}^n F_{\xi_i \xi_j}(x, Du) \eta_{x_i} \eta_{x_j} u_{x_s}^2\right \}^{1/2} \, \left \{ \sum_{i,j,s =1}^n F_{\xi_i \xi_j}(x, Du) u_{x_s x_i} \, u_{x_s x_j}\right \}^{1/2} \, dx \cr\cr
&\le &  C   \int_{\Omega}  (1 + |Du|^2)^{\gamma}  \sum_{i,j,s =1}^n F_{\xi_i \xi_j}(x, Du) \eta_{x_i} \eta_{x_j} u_{x_s}^2 \, dx\cr\cr
&&  + \frac{1}{2} \int_{\Omega} \eta^2  (1 + |Du|^2)^{\gamma}  \sum_{i,j,s =1}^n F_{\xi_i \xi_j}(x, Du) u_{x_s x_i} \, u_{x_s x_j}  \, dx \cr\cr
%&\le &  C (\Lambda)  \int_{\Omega} \Lambda  (1 + |Du|^2)^{\frac{p-2}{2}+\gamma}  \sum_{i,j,s =1}^n \eta_{x_i} \eta_{x_j} u_{x_s}^2 \, dx \\
%&& + \frac{1}{4} \int_{\Omega} \eta^2  (1 + |Du|^2)^{\gamma}  \sum_{i,j,s =1}^n f_{\xi_i \xi_j}(x, Du) u_{x_s x_i} \, u_{x_s x_j} \, dx \\
%&\le &  C(n, \Lambda)  \int_{\Omega} |D\eta|^2  (1 + |Du|^2)^{\frac{p-2}{2}+\gamma}  \sum_{s =1}^n u_{x_s}^2 \, dx \\
%&& + \frac{1}{4} \int_{\Omega} \eta^2  (1 + |Du|^2)^{\gamma}  \sum_{i,j,s =1}^n f_{\xi_i \xi_j}(x, Du) u_{x_s x_i} \, u_{x_s x_j} \, dx \\
&\le& C \int_{\Omega} |D \eta|^2 \,  (1 + |Du|^2)^{\frac{q}{2}+\gamma}  \, dx \cr\cr
&& + \frac{1}{2} \int_{\Omega} \eta^2  (1 + |Du|^2)^{\gamma}  \, \sum_{i,j,s =1}^n F_{\xi_i \xi_j}(x, Du) u_{x_j x_s} u_{x_i x_s} \, dx.
\end{eqnarray}

We can estimate the fourth and the fifth term by the Cauchy-Schwarz and the Young inequalities, together with (F3), as follows
\begin{eqnarray}\label{I4}
|{I}_4| &=& 2 \int_{\Omega} \eta  (1 + |Du|^2)^{\gamma}  \sum_{i,s =1}^n F_{\xi_i x_s}(x, Du) \eta_{x_i} u_{x_s} \, dx\cr\cr
&{\le}&  2 \,  \int_{\Omega} \eta h(x) \,  (1 + |Du|^2)^{\gamma+\frac{q-1}{2}}  \sum_{i,s=1}^n |\eta_{x_i} u_{x_s} | \, dx \cr\cr
&\le&   C \,  \int_{\Omega} \eta |D\eta| |Du| \, h(x) \, (1 + |Du|^2)^{\gamma+\frac{q-1}{2}}   \, dx \cr\cr
&\le& C \int_{\Omega} \eta|D\eta| h(x) (1 + |Du|^2)^{\gamma+\frac{q}{2}}\,dx\cr\cr
&\le& C\int_{\Omega} |D\eta|^2  (1 + |Du|^2)^{\gamma+\frac{p}{2}}\,dx+C\int_{\Omega} \eta^2 h^2(x) (1 + |Du|^2)^{\frac{2q-p}{2}+\gamma}\,dx.
\end{eqnarray}
Moreover
\begin{eqnarray}\label{I5}
|{I}_5| &=& \left| \int_{\Omega} \eta^2 (1 + |Du|^2)^{\gamma} \sum_{i,s=1}^n F_{\xi_i x_s }(x, Du)  u_{x_s x_i} \, dx \right| \cr\cr
&{\le}&  \int_{\Omega} \eta^2\,h(x) (1 + |Du|^2)^{\gamma+\frac{q-1}{2}}   \sum_{i,s=1}^n  u_{x_s x_i} \, dx \cr\cr
&\le&\,  \, \int_{\Omega} \eta^2 h(x) (1 + |Du|^2)^{\gamma+\frac{q-1}{2}}  |D^2 u| \, dx \cr\cr
&=& \int_{\Omega} \left[\eta^2 (1+|Du|^2)^{\frac{p-2}{2}+\gamma} |D^2 u|^2 \right]^{\frac{1}{2}} \left[\eta^2 (1+|Du|^2)^{\frac{2q - p}{2}+\gamma} h^2(x)\right]^{\frac{1}{2}} \, dx \cr\cr
&\le& \varepsilon \int_{\Omega} \eta^2 (1 + |Du|^2)^{\frac{p-2}{2}+\gamma}  |D^2 u|^2 \, dx + C_\varepsilon  \int_{\Omega} \eta^2 h^2(x)(1 + |Du|^2)^{\frac{2q - p}{2}+\gamma}  \, dx,
\end{eqnarray}
where $\varepsilon>0$ will be chosen later.
Finally, similar arguments give
\begin{eqnarray}\label{I6}
|{I}_6| &=&  2 \gamma \, \int_{\Omega} \sum_{i,s=1}^n F_{\xi_i x_s}(x, Du) \eta^2  (1 + |Du|^2)^{\gamma-1}  |Du| D_{x_i}(|Du|) u_{x_s} \, dx\cr\cr
%&=&2 \gamma \, \int_{\Omega} \eta^2 (1 + |Du|^2)^{\gamma-1}  |Du| \sum_{i,s=1}^n f_{\xi_i x_s}(x, Du)  D_{x_i}(|Du|) u_{x_s} \, dx\\
&\le& 2 \, \gamma \, \int_{\Omega} \eta^2 (1+|Du|^2)^{\gamma-\frac{1}{2}}  \sum_{i,s=1}^n F_{\xi_i x_s}(x, Du)  D_{x_i}(|Du|) u_{x_s} \, dx\cr\cr
&\le& 2 \, \gamma \, \int_{\Omega} \eta^2  (1 + |Du|^2)^{\gamma-\frac{1}{2}+\frac{q-1}{2}} \, h(x) \, \sum_{i,s=1}^n D_{x_i}(|Du|) u_{x_s} \, dx\cr\cr
%&\le& 2 \, n \, \gamma \, \Lambda \, \int_{\Omega} \eta^2 (1 + |Du|^2)^{\gamma-\frac{1}{2}+\frac{q-1}{2}}  |D(|Du|)| |Du| \, dx\\
&\le& C \, \gamma \, \int_{\Omega} \, \eta^2 (1 + |Du|^2)^{\gamma+\frac{q-1}{2}}  |D^2 u| \, h(x)\, dx\cr\cr
%&=& C \gamma\int_{\Omega} \left( \eta  h(x) (1+|Du|^2)^{\frac{2q-p}{4}+\frac{\gamma}{2}} \right)  \left( \eta |D^2u| (1+|Du|^2)^{\frac{p-2}{4}+\frac{\gamma}{2}} \right) \, dx\cr\cr
&\le& \varepsilon \int_{\Omega} \eta^2 |D^2u|^2 (1 + |Du|^2)^{\frac{p-2}{2}+\gamma}   \, dx + C_\varepsilon \gamma^2 \,  \, \int_{\Omega} \eta^2  h^2(x)(1 + |Du|^2)^{\frac{2q -p}{2}+\gamma }  \, dx,
\end{eqnarray}
where all the constants $C$ and $C_{\varepsilon}$ are independent of $\gamma$.
\\
Let us now deal with the terms containing the function $g$.   We have
\begin{eqnarray}\label{I789}
|I_7| + |I_8| + |I_9| &\le& 2\int_{\Omega} |g| \,  \eta \, {|D \eta|} \, (1 + |Du|^2)^{\gamma} \, |Du| \, dx  + \gamma \, \int_{\Omega} |g| \, (1 + |Du|^2)^{\gamma - 1} \, D (|Du|^2) \, |Du| \, \eta^2 \, dx\cr\cr
&& + \int_{\Omega} |g| \, (1 + |Du|^2)^{\gamma} |D^2 u| \, \eta^2 \, dx \cr\cr
&\le& 2 \int_{\Omega}\, |g| \, \eta \, {|D \eta|} \, (1 + |Du|^2)^{\gamma} \, |Du| + 2 \gamma \, \int_{\Omega} \, |g| (1 + |Du|^2)^{\gamma - 1} |Du|^2 \, |D^2 u| \, \eta^2 \, dx \cr\cr
&& + \int_{\Omega} |g| (1 + |Du|^2)^{\gamma} |D^2 u| \, \eta^2 \, dx \cr\cr
&\le& \, \int_{\Omega} 2 \, |g| \, \eta \, {|D \eta |} \, (1 + |Du|^2)^{\gamma + \frac{1}{2}}  \, dx + C (1+\gamma) \int_{\Omega} |g| \, \eta^2 \, (1 + |Du|^2)^{\gamma} |D^2 u|  \, dx \cr\cr
&=:& A + B.
\end{eqnarray}
We  estimate the term $A$ working as we did for $I_4$, thus getting
\begin{eqnarray*}
A &\le& \, C \, \int_{\Omega} \eta|D \eta| \, |g(x)| (1 + |Du|^2)^{\gamma + \frac{q}{2}} \, dx\cr\cr
&\le& C\int_{\Omega} |D\eta|^2  (1 + |Du|^2)^{\gamma+\frac{p}{2}}\,dx+C\int_{\Omega} \eta^2 |g(x)|^2 (1 + |Du|^2)^{\frac{2q-p}{2}+\gamma}\,dx	
\end{eqnarray*}
while the term $B$ can be controlled acting as in the estimate of $I_6$, i.e.
\begin{eqnarray*}
B &\le \, & C (1+\gamma) \, \int_{\Omega} |g(x)| \, \eta^2 \, (1 + |Du|^2)^{\gamma} \, |D^2 u| \, dx \\
&\le \, & C (1+\gamma) \, \int_{\Omega} |g(x)| \, \eta^2 \, (1 + |Du|^2)^{\gamma + \frac{q-1}{2}} \, |D^2 u| \, dx \\
&\le& \varepsilon \int_{\Omega} \eta^2 |D^2u|^2 (1 + |Du|^2)^{\frac{p-2}{2}+\gamma}   \, dx + C_\varepsilon (1+\gamma^2)  \, \int_{\Omega} \eta^2  |g(x)|^2(1 + |Du|^2)^{\frac{2q -p}{2}+\gamma }  \, dx.
\end{eqnarray*}}
Now, inserting the estimates of $A$ and $B$ in \eqref{I789} we get
\begin{eqnarray}\label{I7890}
&&|I_7| + |I_8| + |I_9| \le \varepsilon \int_{\Omega} \eta^2 |D^2u|^2 (1 + |Du|^2)^{\frac{p-2}{2}+\gamma}   \, dx \cr\cr
&&+ C (1+\gamma^2) \,  \int_{\Omega} \eta^2  |g(x)|^2(1 + |Du|^2)^{\frac{2q -p}{2}+\gamma }  \, dx+C\int_{\Omega} |D\eta|^2  (1 + |Du|^2)^{\gamma+\frac{p}{2}}\,dx.
\end{eqnarray}

Plugging  \eqref{I1}, \eqref{I4}, \eqref{I5}, \eqref{I6}, \eqref{I7890} in \eqref{ristart} we obtain

\begin{eqnarray*}
&& \int_{\Omega} \eta^2 (1 +|Du|^2)^{\gamma}{\sum_{i,j,s=1}^n} F_{\xi_i \xi_j}(x, Du) u_{x_j x_s}  u_{x_s x_i} \, dx\cr\cr
&\le&  \frac{1}{2} \int_{\Omega} \eta^2  (1 + |Du|^2)^{\gamma}  \, \sum_{i,j,s =1}^n F_{\xi_i \xi_j}(x, Du) u_{x_j x_s} u_{x_i x_s} \, dx\cr\cr
&&+3\varepsilon \int_{\Omega} \eta^2 (1 + |Du|^2)^{\frac{p-2}{2}+\gamma}  |D^2 u|^2 \, dx \cr\cr
&&+ C_\varepsilon (1+\gamma^2)  \, \int_{\Omega} \eta^2  (h^2(x)+|g(x)|^2)(1 + |Du|^2)^{\frac{2q -p}{2}+\gamma }  \, dx\cr\cr
&&+C_\varepsilon \int_{\Omega} |D \eta|^2 \,  (1 + |Du|^2)^{\frac{q}{2}+\gamma}  \, dx.
\end{eqnarray*}
Reabsorbing the first integral in the right hand side by the left hand side we get
\begin{eqnarray*}
&& \frac{1}{2}\int_{\Omega} {\sum_{i,j,s=1}^n} F_{\xi_i \xi_j}(x, Du) u_{x_j x_s} \eta^2 (1 +|Du|^2)^{\gamma} u_{x_s x_i} \, dx\cr\cr
 &\le &3\varepsilon \int_{\Omega} \eta^2 (1 + |Du|^2)^{\frac{p-2}{2}+\gamma}  |D^2 u|^2 \, dx \cr\cr
&&+ C_\varepsilon (1+\gamma^2) \,  \int_{\Omega} \eta^2  (h^2(x)+|g(x)|^2)(1 + |Du|^2)^{\frac{2q -p}{2}+\gamma }  \, dx\cr\cr
&&+C_\varepsilon \int_{\Omega} |D \eta|^2 \,  (1 + |Du|^2)^{\frac{q}{2}+\gamma}  \, dx.
\end{eqnarray*}
Using assumption (F2) in the left hand side of previous estimate, we obtain
\begin{eqnarray*}
&& \frac{\tilde \nu}{2}\int_{\Omega} \eta^2 (1 + |Du|^2)^{\frac{p-2}{2}+\gamma}  |D^2 u|^2 \, dx \, dx\cr\cr
 &\le &3\varepsilon \int_{\Omega} \eta^2 (1 + |Du|^2)^{\frac{p-2}{2}+\gamma}  |D^2 u|^2 \, dx \cr\cr
&&+ C_\varepsilon (1+\gamma^2) \, \int_{\Omega} \eta^2  (h^2(x)+|g(x)|^2)(1 + |Du|^2)^{\frac{2q -p}{2}+\gamma }  \, dx\cr\cr
&&+C_\varepsilon \int_{\Omega} |D \eta|^2 \,  (1 + |Du|^2)^{\frac{q}{2}+\gamma}  \, dx
\end{eqnarray*}
Choosing $\varepsilon=\frac{\tilde \nu}{12}$, we can reabsorb the first integral in the right hand side by the left hand side thus getting
\begin{eqnarray*}
&& \int_{\Omega} \eta^2 (1 + |Du|^2)^{\frac{p-2}{2}+\gamma}  |D^2 u|^2 \, dx \, dx\cr\cr
 &\le & C (1+\gamma^2)  \, \int_{\Omega} \eta^2  (h^2(x)+|g(x)|^2)(1 + |Du|^2)^{\frac{2q -p}{2}+\gamma }  \, dx\cr\cr
&&+C \int_{\Omega} |D \eta|^2 \,  (1 + |Du|^2)^{\frac{q}{2}+\gamma}  \, dx\cr\cr
&\le & C (1+\gamma^2)  \, \int_{\Omega} (\eta^2 +|D \eta|^2 ) (1+h^2(x)+|g(x)|^2)(1 + |Du|^2)^{\frac{2q -p}{2}+\gamma }  \, dx.
\end{eqnarray*}
Using the assumptions on $h$ and $g$,
 the properties of $\eta$ and H\"older's inequality, we arrive at
\begin{eqnarray}\label{pinca}
\int_{B_{t_2}} \eta^2 (1 + |Du|^2)^{\frac{p-2}{2} + \gamma}  |D^2 u|^2 \, dx \le  C  (1 + \gamma^2)\frac{\Theta}{(t_2 - t_1)^{2}} \left [\int_{B_{t_2}}(1 + |Du|^2)^{\frac{(2\gamma + 2q - p)m}{2}}  \, dx \right
]^{\frac{1}{m}}
%&& + \textcolor{verde}{C(\gamma)} \|g\|^2_{L^r(\Omega)} \, \left (\int_{B_R}  (1 + |Du|^2)^{\frac{(p + 2 \gamma) n}{2 (n-2)}} \, dx  \right )^{\frac{(n-2) (2 \gamma - p + 2)}{n (p + 2 \gamma)}}
\end{eqnarray}
for any  $0 < t_1 < t_2$, where the constant $C$  is independent of $\gamma$, and where we set
\[
\Theta = (1 + \|g\|^2_{L^r(\Omega)} + \|h\|^2_{L^r(\Omega)})
\]
and $$m=\frac{r}{r-2}.$$
The Sobolev  inequality yields
\begin{eqnarray*}
&&\left ( \int_{B_{t_2}}  \eta^{2^*}(1 + |Du|^2)^{(\frac{p}{4} + \frac{\gamma}{2}) 2^*}\, dx\right )^{\frac{2}{2^*}} \le \, C \, \int_{B_{t_2}} |D (\eta(1 + |Du|^2)^{\frac{p}{4} + \frac{\gamma}{2}})|^2 \, dx 	\cr\cr
&\le&C(1+\gamma^2)\int_{B_{t_2}} \eta^2 (1 + |Du|^2)^{\frac{p-2}{2} + \gamma}  |D^2 u|^2\,dx+C \int_{B_{t_2}} |D \eta|^2 \,  (1 + |Du|^2)^{\frac{p}{2}+\gamma}  \, dx,
\end{eqnarray*}
where we set
$$ 2^*=\begin{cases}
	\displaystyle{\frac{2n}{n-2}}\qquad\qquad\qquad\qquad\qquad \text{if}\,\,n\ge3\cr\cr
	\text{any finite exponent}\, \qquad\qquad \text{if}\,\,n=2.
	\end{cases}$$
	Using estimate \eqref{pinca} to control the first integral in the right hand side of previous inequality, we obtain
\begin{eqnarray}\label{stimapreite}
&&\left ( \int_{B_{t_2}}  \eta^{2^*}(1 + |Du|^2)^{(\frac{p}{4} + \frac{\gamma}{2}) 2^*}\, dx\right )^{\frac{2}{2^*}} 	\cr\cr
&\le& C\frac{\Theta(1+\gamma^4)}{(t_2 - t_1)^{2}}\left [\int_{B_{t_2}}(1 + |Du|^2)^{\frac{(2\gamma + 2q - p)m}{2}}  \, dx \right ]^{\frac{1}{m}}\cr\cr
&&+\frac{C}{(t_2 - t_1)^{2}} \int_{B_{t_2}}  \,  (1 + |Du|^2)^{\frac{p}{2}+\gamma}  \, dx\cr\cr
&\le& C\frac{\Theta(1+\gamma^4)}{(t_2 - t_1)^{2}}\left [\int_{B_{t_2}}(1 + |Du|^2)^{\frac{(2\gamma + 2q - p)m}{2}}  \, dx \right ]^{\frac{1}{m}},
\end{eqnarray}	
where we used that $p\le 2q-p$ and that $L^m\hookrightarrow L^1$.
\\
By virtue of Theorem \ref{Chiara} and by Sobolev embedding we have that $Du\in L^{\frac{np}{n-2}}_{\mathrm{loc}}(\Omega)$. However, in the sequel, we need a quantitative estimate of the $ L^{\frac{np}{n-2}}$ norm of $Du$ that
takes into account the explicit dependencies on the structural parameters.
\\
{To this aim,  we remark that for $\gamma=0$ previous estimate reads as
\begin{eqnarray}\label{stimapreite0}
&&\left ( \int_{B_{t_2}}  \eta^{2^*}(1 + |Du|^2)^{\frac{p}{2} \cdot \frac{2^*}{2}}\, dx\right )^{\frac{2}{2^*}} 	\cr\cr
&\le& C\frac{\Theta}{(t_2-t_1)^{2}}\left [\int_{B_{t_2}}(1 + |Du|^2)^{\frac{( 2q - p)m}{2}}  \, dx \right ]^{\frac{1}{m}}
\end{eqnarray}
and, since {by virtue of the assumption \eqref{gap}
\begin{equation}
\label{cons-app-2}
(2q-p)m<p\frac{2^*}{2}
\end{equation}
(see Section \ref{appendix})}, there exists $\vartheta\in (0,1)$ such that
\begin{equation}\label{vartheta}
\frac{1}{( 2q-p) m} = \frac{\vartheta}{ p} + \frac{1 - \vartheta}{ \frac{2^* p}{2}}.
\end{equation}
Actually, one can easily check that
\[
\vartheta = 1-\frac{n}{2}\left(1-\frac{p}{(2q-p)m}\right)
\]
and
\begin{equation}
\label{def-theta-2}
1 - \vartheta = \frac{n}{2}\left(1-\frac{p}{(2q-p)m}\right).
\end{equation}
So we can use the following interpolation inequality
\begin{equation}
\label{interpolation-gamma}
\|(1+|Du|^2)^{\frac{1}{2}}\|_{L^{( 2q-p) m}(B_{t_2})} \le \, \|(1+|Du|^2)^{\frac{1}{2}}\|^{\vartheta}_{L^{ p}(B_{t_2})} \, \|(1+|Du|^2)^{\frac{1}{2}}\|^{1 - \vartheta}_{L^{\frac{2^*p}{2}}(B_{t_2})}
\end{equation}
to   estimate the first integral on the right hand side of\eqref{stimapreite0}, thus  obtaining
\begin{eqnarray}\label{preite}
&& \left ( \int_{B_{t_1}}  (1 + |Du|^2)^{\frac{p  n}{2 (n-2)}}\, dx \right )^{\frac{n-2}{n}} \cr\cr
&\le &\, C   \, \frac{\Theta}{(t_2 - t_1)^{2}} \left [\int_{B_{t_2}}  \,  (1 + |Du|^2)^{\frac{p  n}{2 (n-2)}} \, dx \right ]^{\frac{( 2 q - p)}{\frac{2^*p}{2}} (1 - \vartheta)}\cr\cr
&& \times \left [\int_{B_{t_2}}  \, (1 + |Du|^2)^{\frac{ p}{2}}  \, dx \right ]^{\frac{( 2 q - p)}{ p} \vartheta}.
\end{eqnarray}
Observe that
\[
\frac{2q -p }{ p} (1 - \vartheta)={\frac{n}{2}\left(\frac{ 2q -p }{ p}-\frac{1}{m}\right)} < 1.
\]
Indeed
$$\frac{n}{2}\left(\frac{2q -p }{ p}-\frac{1}{m}\right)<1\,\,\Longleftrightarrow\,\, \frac{2q }{ p}< 1+\frac{1}{m}+\frac{2}{n}=2+\frac{1}{m}-\frac{2}{2^*},$$
which is equivalent to \eqref{gap}.
Therefore, we are legitimate to Young's inequality  {with exponents $\frac{p}{(2q-p)(1-\vartheta)}$ and $\frac{p}{p-(2q-p)(1-\vartheta)}$ in the right hand side of \eqref{preite}} thus getting
\begin{eqnarray*}
&& \left ( \int_{B_{t_1}}  (1 + |Du|^2)^{\frac{pn}{2 (n-2)}}\, dx \right )^{\frac{n-2}{n}}\le  \varepsilon  \left [\int_{B_{t_2}}  \,  (1 + |Du|^2)^{\frac{p n}{2 (n-2)}} \, dx \right ]^{\frac{n - 2}{n}}\cr\cr
&& +C \, \left( \frac{\Theta }{(t_2 - t_1)^2}\right)^{\frac{p}{p-(2q-p)(1-\vartheta)}}
\left [\int_{B_{t_2}}  \, (1 + |Du|^2)^{\frac{p}{2}}  \, dx \right ]^{\frac{( 2 q - p)\vartheta}{( 2 q - p)\vartheta-2(q-p)}}.
\end{eqnarray*}
Since previous estimate holds true for every $\rho<t_1<t_2<R$, Lemma \ref{holf}
implies that
\begin{eqnarray*}
 \left ( \int_{B_{t_1}}  (1 + |Du|^2)^{\frac{pn}{2 (n-2)}}\, dx \right )^{\frac{n-2}{n}}\le  C \, \left( \frac{\Theta }{(t_2 - t_1)^2}\right)^{\frac{p}{p-(2q-p)(1-\vartheta)}}
\left [\int_{B_{t_2}}  \, (1 + |Du|^2)^{\frac{p}{2}}  \, dx \right ]^{\frac{( 2 q - p)\vartheta}{( 2 q - p)\vartheta-2(q-p)}}
\end{eqnarray*}
 and, since $r>n$, also that
\begin{equation}\label{primopasso}
	 \left ( \int_{B_{\rho}}  (1 + |Du|^2)^{\frac{pm}{2}}\, dx \right )^{\frac{1}{m}}\le  C \, \left( \frac{\Theta }{(R - \rho)^2}\right)^{\frac{p}{p-(2q-p)(1-\vartheta)}}
\left [\int_{B_R}  \, (1 + |Du|^2)^{\frac{p}{2}}  \, dx \right ]^{\frac{( 2 q - p)\vartheta}{( 2 q - p)\vartheta-2(q-p)}}.
\end{equation}
Now, setting $$V(Du)=(1+|Du|^2)^{\frac{1}{2}},$$ we can write \eqref{stimapreite} as follows
\begin{eqnarray}\label{stimapreitebis}
&&\left ( \int_{B_{R}}  \eta^{2^*}V(Du)^{(p+2\gamma)\frac{2^*}{2}}\, dx\right )^{\frac{2}{2^*}} 	\cr\cr
&\le& C\frac{\Theta (1+\gamma^4)}{(R - \rho)^{2}}\left [\int_{B_R}V(Du)^{[ 2(q - p)m+(p+2\gamma)m]}  \, dx \right ]^{\frac{1}{m}} \cr\cr
&\le& C\frac{\Theta (1+\gamma^4)}{(R - \rho)^{2}}||V(Du)||^{2(q-p)}_{L^\infty(B_R)}\left [\int_{B_R}V(Du)^{(p+2\gamma)m}  \, dx \right ]^{\frac{1}{m}}
\end{eqnarray}
and so
\begin{eqnarray}\label{stimapreiteter}
&&\left ( \int_{B_{\rho}}  V(Du)^{[m(p+2\gamma)]\frac{2^*}{2m}}\, dx\right )^{\frac{2m}{2^*}} 	\cr\cr
&\le& C {\left [\frac{\Theta (1+\gamma^4)}{(R - \rho)^{2}} \right ]^m} ||V(Du)||^{2m(q-p)}_{L^\infty(B_R)}\int_{B_R}V(Du)^{(p+2\gamma)m}  \, dx ,
\end{eqnarray}	
where we also used that $\eta=1$ on $B_\rho$.
\\
{With radii $\frac{R_0}{2}\le \bar \rho<\bar R\le R_0$ fixed at the beginning of the section, we} define the decreasing sequence of radii by setting
$$\rho_i=\bar \rho+\frac{\bar R-\bar \rho}{2^i}.$$
Let us also define the following increasing sequence of exponents
$$p_0=pm\qquad\qquad {p_{i}}={p_{i-1}}\frac{2^*}{2m}=p_0\left(\frac{2^*}{2m}\right)^{i}$$
Noticing that, by virtue of \eqref{primopasso}, estimate \eqref{stimapreiteter} holds true for $\gamma=0$ and for every $\bar \rho<\rho<R<\bar R$, we may iterate it  on the concentric balls $B_{\rho_i}$  with exponents $p_i$,
thus obtaining
\begin{eqnarray}\label{stimapreitepassoi}
&&\left ( \int_{B_{\rho_{i+1}}}  V(Du)^{p_{i+1}}\, dx\right )^{\frac{1}{p_{i+1}}} 	\cr\cr
&\le& \displaystyle{\prod_{j=0}^{i}}\left(C {\frac{\Theta^m p_j^{4m}}{(\rho_j - \rho_{j+1})^{2m}}}||V(Du)||^{2m(q-p)}_{L^\infty(B_R)}\right)^{\frac{1}{p_j}}\left(\int_{B_{\rho_0}}V(Du)^{p_0}  \,
dx\right)^{\frac{1}{p_0}}\cr\cr
&=& \displaystyle{\prod_{j=0}^{i}}\left(C{\frac{4^{j+1}\Theta^m p_j^{4m}}{(\bar R-\bar \rho )^{2m}}}||V(Du)||^{2m(q-p)}_{L^\infty(B_R)}\right)^{\frac{1}{p_j}}\left(\int_{B_{\rho_0}}V(Du)^{p_0}  \,
dx\right)^{\frac{1}{p_0}}\cr\cr
&=& \displaystyle{\prod_{j=0}^{i}}\left(4^{j+1} {p_j^{4m}} \right)^{\frac{1}{p_j}}\displaystyle{\prod_{j=0}^{i}}\left({\frac{C\Theta^m }{(\bar R-\bar \rho
)^{2m}}}||V(Du)||^{2m(q-p)}_{L^\infty(B_R)}\right)^{\frac{1}{p_j}}\left(\int_{B_{\rho_0}}V(Du)^{p_0}  \, dx\right)^{\frac{1}{p_0}}.
\end{eqnarray}
Since
$$\displaystyle{\prod_{j=0}^{i}}\left(4^{j+1} {p_j^{4m}}\right)^{\frac{1}{p_j}}=\exp\left(\sum_{j=0}^i \frac{1}{p_j}\log(4^{j+1} {p_j^{4m}})\right)\le \exp\left(\sum_{j=0}^{+\infty} \frac{1}{p_j}\log({4}^{j+1}
{p_j^{4m}})\right) \le c(m)$$
and
\begin{eqnarray*}
	&&\displaystyle{\prod_{j=0}^{i}}\left(\frac{C {\Theta^m} }{{(\bar R-\bar \rho )^{2m}}}||V(Du)||^{2m(q-p)}_{L^\infty(B_R)}\right)^{\frac{1}{p_j}}=\left(\frac{C {\Theta^m} }{{(\bar R-\bar \rho
)^{2m}}}||V(Du)||^{2m(q-p)}_{L^\infty(B_R)}\right)^{\sum_{j=0}^{i}\frac{1}{p_j}}
	\cr\cr
	&\le& \left(\frac{C {\Theta^m} }{{(\bar R-\bar \rho )^{2m}}}||V(Du)||^{2m(q-p)}_{L^\infty(B_R)}\right)^{\sum_{j=0}^{+\infty}\frac{1}{p_j}}=\left(\frac{C {\Theta^m} }{{(\bar R-\bar \rho
)^{2m}}}||V(Du)||^{2m(q-p)}_{L^\infty(B_R)}\right)^{\frac{n(r-2)}{2p_0(r-n)}}
\end{eqnarray*}
we can let $i\to \infty$ in \eqref{stimapreitepassoi} thus getting
\begin{eqnarray}\label{stimaquasifin}
	||V(Du)||_{L^{\infty}(B_{\bar\rho})} \le C \left(\frac{\Theta }{(\bar R-\bar \rho )^{2}}\right)^{{\frac{n(r-2)}{2p(r-n)}}}||V(Du)||_{L^{\infty}(B_{\bar R})}^{\frac{n(q-p)(r-2)}{p(r-n)}}\left(\int_{B_{\bar R}}V(Du)^{p_0}
\, dx\right)^{\frac{1}{p_0}},
\end{eqnarray}
where we used that  $p_0=pm$.  Now we have that
 $$\frac{n(q-p)(r-2)}{p(r-n)}<1\,\,\Longleftrightarrow\,\,\frac{q-p}{p}< \frac{r-n}{n(r-2)}$$
{which holds because, assumption \eqref{gap} implies that} $$\frac{q-p}{p}< \frac{1}{n}-\frac{1}{r}=\frac{r-n}{rn}<\frac{r-n}{n(r-2)}$$
 Therefore, by the use of Young's inequality  with exponents $\frac{p(r-n)}{n(q-p)(r-2)}$ and $\frac{p(r-n)}{p(r-n)-n(q-p)(r-2)}$ in the right hand side of \eqref{stimaquasifin}, we get
 \begin{eqnarray}\label{stimaquasifin00}
	||V(Du)||_{L^{\infty}(B_{\bar\rho})} \le \frac{1}{2} ||V(Du)||_{L^{\infty}(B_{\bar R})}+\left(\frac{C\Theta }{(\bar R-\bar \rho )^{2}}\right)^\sigma\left(\int_{B_{\bar R}}V(Du)^{pm}  \, dx\right)^{\frac{\sigma}{pm}},
\end{eqnarray}
for an exponent ${\sigma=\sigma(n,r,p,q)}>0$. Since previous estimate holds true for every $\frac{R_0}{2}<\bar\rho<\bar R<R_0$, by Lemma \ref{holf} we get\begin{eqnarray}\label{stimaquasifin3}
	||V(Du)||_{L^{\infty}\left(B_{\frac{R_0}{2}}\right)} \le \left(\frac{C\Theta }{R_0^{2}}\right)^\sigma\left(\int_{B_{ R_0}}V(Du)^{pm}  \, dx\right)^{\frac{\sigma}{pm}}.
\end{eqnarray}
Estimate \eqref{stimafinale}} follows combining \eqref{stimaquasifin3} with \eqref{primopasso}.

\begin{remark}
It is worth noticing that in the case $p = q$ we retrieve the following result
\begin{corollary}
\label{corollario-secondo}
Let $u\in \mathcal{K}_\psi(\Omega) \cap W^{1,\infty}_{\mathrm{loc}}(\Omega)$ be a solution to {\eqref{obst-def0}} under the assumptions \textnormal{(F1)--(F3)} with $p = q$ and $h \in L^{\infty}_{\rm loc}(\Omega)$.
 If $\psi\in W^{2,r}_{\mathrm{loc}}(\Omega)$, for some $r > n$, then the following a priori estimate
 \begin{eqnarray*}
 	\sup_{B_{R/2}}|Du|  \le \left(\frac{C }{R^{2}}\right)^\sigma\left(\int_{B_{ R}}(1+{|Du|})^{p}  \, dx\right)^{\frac{\sigma}{p}}
 \end{eqnarray*}
 holds for every ball $B_R\Subset\Omega$, for $\sigma(n, r, p)>0$ and with $C=C(||h||_{\infty}, ||D^2\psi||_r,\tilde\nu_1, \tilde L_1,p,n)$ but independent of $||Du||_{L^\infty}$.		
\end{corollary}
\end{remark}

\medskip

{\bf Step 2. Conclusion.}\quad Now we conclude the proof by passing to the limit in the approximating problem. The limit procedure is standard { see for example (\cite{CGGP})} and we just give here a brief sketch.
\\
{Let $u \in \mathcal{K}_{\psi}(\Omega)$ be a solution to \eqref{obst-def0} and } let ${F^\varepsilon}$ be the sequence obtained applying Lemma \ref{apprcupguimas2parte}  to the integrand $F$. Let us fix a ball
$B_R\Subset\Omega$ and let ${u^\varepsilon}\in u+W^{1,p}_0(B_R)$ be the solution to the minimization problem
$$\min\left\{\int_{B_R} {F^\varepsilon}(x,Dz):\,\, z\in \mathcal{K}_{\psi}(B_R)\right\}.$$
%\tcr{where $\psi_\varepsilon$ is a standard mollification of $\psi$}.
By Theorem \ref{CarloMichelebis}, the minimizers ${u^\varepsilon}$ satisfy the  a priori assumptions at \eqref{regE}, i.e. ${u^\varepsilon}\in W^{1,\infty}_{\mathrm{loc}}(\Omega)$, and therefore we are legitimate to use
estimate \eqref{stimafinale} thus obtaining
 \begin{eqnarray}\label{stimaquasifin33}
	||V(D{u^\varepsilon})||_{L^{\infty}\left(B_{\frac{R_0}{2}}\right)} \le \left(\frac{C\Theta }{R_0^{2}}\right)^\sigma\left(\int_{B_{ R_0}}V(D{u^\varepsilon})^{p}  \, dx\right)^{\frac{\sigma}{p}}.
\end{eqnarray}
 By the first inequality of growth conditions  at (III) of Lemma \ref{apprcupguimas2parte} and the minimality of $u_\varepsilon$ we get
\begin{eqnarray*}
\int_{B_R} |D{u^\varepsilon}(x)|^p\,dx &\le& {C(K_0)}\int_{B_R} {F^\varepsilon}(x,D {u^\varepsilon}(x))\,dx \cr\cr
&\le& {C(K_0)} \int_{B_R} {F^\varepsilon}(x,Du(x))\,dx\cr\cr
&\le& {C(K_0)}\int_{B_R} F(x,Du(x)) \, dx \,,
\end{eqnarray*}
where in the last estimate we used the second inequality at (I) of Lemma \ref{apprcupguimas2parte}.
Since $F(x, Du)\in L^1_{\mathrm{loc}}(\Omega)$ by assumption,   we deduce, up to subsequences, that there exists $\bar u\in W^{1,p}_0(B_R)+u$ such that $${u^\varepsilon}\rightharpoonup\bar{u}\quad\quad\text{weakly in
}W^{1,p}_0(B_R)+u\,.$$ Note that, since $u^\varepsilon\in \mathcal{K}_\psi$ for every $\varepsilon$ and $\mathcal{K}_\psi$ is a closed set, we have $\bar u\in \mathcal{K}_\psi$. Our next aim is to show that $\bar u$ is a
solution to our obstacle problem over the ball $B_R$.
To this aim fix $\varepsilon_0>0$ and observe that the lower semicontinuity of the functional $w\mapsto \int_{B_R}F^{\varepsilon_0}(x,Dw)\,dx$, the minimality of $u^\varepsilon$ and the monotonicity of the sequence of
$F^\varepsilon$ yield
\begin{eqnarray*}
\int_{B_R} {F^{\varepsilon_0}}(x,D\bar{u})\,dx\le \lim_{\varepsilon\to 0}\int_{B_R} F^{\varepsilon_0}(x,Du^{\varepsilon})\,dx	\le
\int_{B_R} F^{\varepsilon_0}(x,Du)\,dx\le \int_{B_R} F(x,Du)\,dx
\end{eqnarray*}
%Hence from estimate \eqref{apriori} and by compact embedding we infer
%\begin{eqnarray*}
%V_p(Du_\varepsilon)\rightharpoonup v&&\text{weakly in }W^{1,2}_{\rm loc}(\Omega)\,,\\
%V_p(Du_\varepsilon)\rightarrow v&&\text{strongly in }L^2_{\rm loc}(\Omega)\,,
%\end{eqnarray*}
%from which we deduce, together with inequality \eqref{Vp},
%$$Du_\varepsilon\rightarrow\bar{v}\quad\quad\text{strongly in }L^p_{\rm loc}(\Omega)\,.$$
%We thus have the strong convergence
%$$u_\varepsilon\rightarrow\bar{u}\quad\quad\text{in }W^{1,p}_0(B_R)+u\,,$$
%and the limit function $\bar{u}$ still belongs to $\mathcal{K}_\psi(\Omega)$ since this set is closed. It remains to prove that $\bar{u}$ minimizes the functional $\int_\Omega F(x,Dw)$ in $W^{1,p}_0(B_R)+u$.
%Now, fix $\varepsilon_0$ and observe that, by (I) in Lemma \ref{apprcupguimas2parte}, it holds
%\begin{equation}\label{conv}
%	\int_{B_R} F_{\varepsilon_0}(x,Du_\varepsilon(x))\,dx\le  \int_{B_R} F_{\varepsilon}(x,Du_\varepsilon(x))\,dx \le \int_{B_R} F(x,Du_\varepsilon(x))\,dx,
%\end{equation} for $\varepsilon>\varepsilon_0$.
%\\
%By the semicontinuity of { of the functional $v\mapsto \int_{B_R}F_{\varepsilon_0}(x,Dv(x)\,dx$ , the weak convergence of $u_\varepsilon$ to $\bar u$ in $u+W^{1,p}(B_R)$ and estimate \eqref{conv}}  it follows
%\tcr{\begin{eqnarray}\label{Fineq}
%\int_{B_R} F_{\tcr{\varepsilon_0}}(x,D\bar{u})\,dx &\le& \liminf_{\varepsilon\to 0}\int_{B_R} F_{{\varepsilon_0}}(x,Du_\varepsilon)\,dx \cr\cr
%&\le&  \int_{B_R} F(x,Du)\,dx
%\end{eqnarray}}
We now use monotone convergence Theorem in the left hand side of previous estimate to deduce that
$$\int_{B_R} F(x,D\bar{u})\,dx=\lim_{\varepsilon_0\to 0}\int_{B_R} {F^{\varepsilon_0}}(x,D\bar{u})\,dx\le \int_{B_R} F(x,Du)\,dx$$
 We have then proved that the limit function $\bar{u}\in W^{1,p}(B_R)+u$ is a solution to the minimization problem
$$\min\left\{\int_\Omega F(x,Dw):w\in W^{1,p}_0(B_R)+u,\,w\in \mathcal{K}_{\psi}\right\}.$$
Since by the strict convexity of the functional the solution is unique, we conclude that $u=\bar u$.
It is quite routine to show that the convergence of ${u^\varepsilon}$ to $u$ is strong in $W^{1,p}_{\mathrm{loc}}(B_R)$ and hence the conclusion follows passing to the limit as $\varepsilon\to 0$ in estimate
\eqref{stimaquasifin33}.
\end{proof}

\section{Appendix.}

\subsection {Some consequences of assumption \eqref{gap}}\label{appendix}

{In this Section we show that our assumption \eqref{gap} implies two facts: first of all that {if we assume that $q < n$ than it turns out that} $2 q - p < r$;
%\tcr{and therefore in the linearization procedure to establish the regularity of the function $g$ we can use directly the results in \cite{G19}};
 second {we show} \eqref{cons-app-2}, which instead have been used in performing the a priori estimate, both related to our Lipschitz regularity result.
\\
\\
First let us prove that $2q - p < r$. Indeed \eqref{gap} is equivalent to
\[
r > \frac{1}{\displaystyle 1 + \frac{1}{n} - \frac{q}{p}}.
\]
If we are able to show that
\begin{equation}
\label{claim-app}
2q - p < \frac{1}{\displaystyle 1 + \frac{1}{n} - \frac{q}{p}},
\end{equation}
then our result follows.
\\
\\
Let us prove \eqref{claim-app}. We have
\begin{eqnarray*}
2q - p < \frac{1}{\displaystyle 1 + \frac{1}{n} - \frac{q}{p}} &\Longleftrightarrow& \frac{np}{np + p - qn} > 2q - p  \Longleftrightarrow np > (2q - p)(np + p - qn) \\
&\Longleftrightarrow& np + (q-p)^2 n + q n (q - p) > p (q - p) + q p
\end{eqnarray*}
which holds true as long as $qn(q-p) > p (q - p)$ and $np > pq$, as long as we assumed $q < n$.
\\
\\
On the other hand let us now prove \eqref{cons-app-2}.  Indeed  \eqref{gap} entails that
\[
2q - p \le \, 2 p \left (1 + \frac{1}{2m} - \frac{1}{2^*} \right ) - p = p \left (1 + \frac{1}{m} - \frac{2}{2^*} \right )
\]
so that
\[
(2q - p) m \le \, p m \left (1 + \frac{1}{m} - \frac{2}{2^*} \right )
\]
and it turns out that
\[
p m \left (1 + \frac{1}{m} - \frac{2}{2^*} \right ) \le \, p \frac{2^*}{2}  \Longleftrightarrow 1 + \frac{1}{m} - \frac{2}{2^*} \le \, \frac{2^*}{2m} \Longleftrightarrow 1 + \frac{1}{m} - \frac{2}{2^*} - \frac{2^*}{2m} \le \,
0
\]
but
\[
\frac{2^*}{2m} > 1 \Rightarrow \frac{2^*}{2m} = 1 + \varepsilon
\]
{for some $\varepsilon > 0$}, therefore
\[
1 + \frac{1}{m} - \frac{2}{2^*} - \frac{2^*}{2m} = 1 + \frac{2}{2^*} + \frac{2}{2^*} \varepsilon - \frac{2}{2^*} - 1 - \varepsilon = \left ( \frac{2}{2^*} - 1\right ) \varepsilon  \le 0
\]
which is what we would like to prove.}

\subsection{{On} the validity of the variational inequality \eqref{obst-def}.}\label{EL}
Let $u\in \mathcal{K}_\psi(\Omega)$ be a solution to {\eqref{obst-def0}} under the assumptions \textnormal{(F1)--(F3)}. Suppose that there exists $r>n$ such that $h\in L^r_{\mathrm{loc}}(\Omega)$, where $h(x)$ is the function
appearing in \textnormal{(F3)}. Assume moreover that  $2\le p\le q$ are such that \eqref{gap} is satisfied.\\
If $\psi\in W^{2,\chi}_{\mathrm{loc}}(\Omega)$ with $\chi\ge 2q-p$,
 %by the remark in Subsection \ref{appendix}, then $\psi\in W^{2,2q-p}_{\mathrm{loc}}(\Omega)$.
 by Theorem \ref{Chiara} and the Sobolev imbedding Theorem we have that $Du\in L^{\frac{np}{n-2}}_{\mathrm{loc}}(\Omega)$.
 Since $u$ is a solution to \eqref{obst-def} then
 $$0\le \int_\Omega \Big[F(x, Du+\varepsilon Dv)-F(x, Du)\Big]\,dx,$$
 for every $v\in W^{1,q}_0(\Omega)$, $v\ge 0$ and for every $\varepsilon>0$.
 Therefore,  we have
 \begin{eqnarray*}
 0&\le& \varepsilon\int_\Omega \int_0^1\langle {F_{\xi}}(x, Du+s\varepsilon Dv),Dv\rangle\,ds\,dx\cr\cr
 &=&\varepsilon\int_\Omega \langle {F_{\xi}} (x, Du),Dv\rangle\,dx+\varepsilon\int_\Omega \int_0^1\langle {F_{\xi}}(x, Du+s\varepsilon Dv)- {F_{\xi}}(x, Du),Dv\,ds\rangle\,dx\cr\cr
 &=&\varepsilon\int_\Omega \langle {F_{\xi}}(x, Du),Dv\rangle\,dx+\varepsilon^2\int_\Omega \int_0^1s\int_0^1\langle {F_{\xi \xi}}(x, Du+st\varepsilon Dv)Dv,Dv\rangle\,dt\,ds\,dx	
 \end{eqnarray*}
 Dividing both side of previous inequality by $\varepsilon$ we have
 \begin{eqnarray}\label{EL1}
 0&\le&
\int_\Omega \langle {F_{\xi}}(x, Du),Dv\rangle\,dx+\varepsilon\int_\Omega \int_0^1s\int_0^1\langle {F_{\xi \xi}}(x, Du+st\varepsilon Dv)Dv,Dv\rangle \,dt\,ds\,dx	\cr\cr
&=:& I+II
 \end{eqnarray}
We observe that
\begin{eqnarray*}
	II&\le& \varepsilon \int_\Omega \int_0^1s\int_0^1| {F_{\xi \xi}}(x, Du+st\varepsilon Dv)|Dv|^2 \,dt\,ds\,dx\cr\cr
	&\le& \varepsilon \int_\Omega \int_0^1s\int_0^1(1+|Du+st\varepsilon Dv|^2)^{\frac{q-2}{2}}|Dv|^2 \,dt\,ds\,dx\cr\cr
	&\le& \varepsilon \int_\Omega (1+|Du|^2+\varepsilon^2| Dv|^2)^{\frac{q-2}{2}}|Dv|^2 \,dx	
\end{eqnarray*}
Since $$q<p\left(1+\frac{1}{n}-\frac{1}{r}\right)<p\frac{n}{n-2}$$ and $$Du\in L^{\frac{np}{n-2}}_{\mathrm{loc}}(\Omega)\hookrightarrow L^q_{\mathrm{loc}}(\Omega)$$
we have that $$\lim_{\varepsilon\to 0}II=0$$
and passing to the limit as $\varepsilon\to 0$ in \eqref{EL1}, we have
$$\int_\Omega \langle {F_{\xi}}(x, Du),Dv\rangle\,dx\ge 0,$$
for every $v\in W^{1,q}_0(\Omega)$, $v\ge 0$. In particular we may choose $v=\varphi-u$ with $\varphi\in W^{1,q}_{\rm loc}(\Omega)$, $\varphi\ge \psi$ in $\Omega$
thus getting
$$\int_\Omega \langle {F_{\xi}}(x, Du),D\varphi-Du\rangle\,dx\ge 0.$$ This result has been used in Subsection \ref{linear}.

\end{document}